%% file: main.tex
\UseRawInputEncoding
\documentclass[a4paper,12pt]{amsart}
\usepackage[utf8]{inputenc}
\usepackage{tikz-cd}
\usepackage[T1]{fontenc}
\usepackage{amssymb}
\usepackage[all]{xy}
\usepackage{graphicx}
\usepackage{amsmath}
\usepackage{amsthm}
\usepackage{float}
\usepackage{makecell}
\usepackage{bookmark}
\usepackage{chngpage}
\usepackage[top=3cm,bottom=3cm,left=3cm,right=3cm]{geometry}
\usepackage{tikzit}
\input{sample.tikzstyles}
\usepackage{arydshln}
\usepackage{colonequals} 
\usepackage{fullpage}
\usepackage{stackengine}

\DeclareMathOperator{\Gal}{Gal}

\DeclareMathOperator{\rank}{rank}

\DeclareMathOperator{\Triv}{Triv}
\DeclareMathOperator{\NS}{NS}

\DeclareMathOperator{\Deltaconic}{\Delta_{\text{conic}}}
\DeclareMathOperator{\Deltaell}{\Delta_{\text{ell}}}

\DeclareMathOperator{\Df}{Df}
\DeclareMathOperator{\Res}{Res}

\title{Conic bundles and Mordell--Weil ranks of elliptic surfaces}
\author{Felipe Zingali Meira}
\address{Universidade Federal do Rio de Janeiro, Rijksuniversiteit Groningen}
\email{f.zingali.meira@rug.nl}
\date{10/2024}

\newtheorem{thm}{Theorem}[section]
\newtheorem{lemma}[thm]{Lemma}
\newtheorem{corollary}[thm]{Corollary}
\newtheorem*{acknowledge}{Acknowledgments}

\newtheorem{prop}[thm]{Proposition}

\newtheorem{conjecture}[thm]{Conjecture}

\theoremstyle{definition}
\newtheorem{dfn}[thm]{Definition}
\newtheorem{remark}[thm]{Remark}
\newtheorem{notation}[thm]{Notation}
\newtheorem{assumption}[thm]{Assumption}
\newtheorem{example}[thm]{Example}

\newcommand{\cB}{\mathcal B}

\newcommand{\cE}{\mathcal E}         
\newcommand{\cF}{\mathcal F}         
\newcommand{\cG}{\mathcal G}

\newcommand{\ZZ}{\mathbb{Z}}
\newcommand{\QQ}{\mathbb{Q}}

\newcommand{\FF}{\mathbb{F}}

\newcommand{\PP}{\mathbb{P}}

\newcommand{\Galk}{\Gal(\overline{k}/k)}

\begin{document}

\maketitle

\begin{abstract}
    Let $k$ be a number field and $\cE$ an elliptic curve defined over the function field $k(T)$ given by an equation of the form $y^2 = a_3x^3 + a_2x^2 + a_1x + a_0$, where $a_i \in k[T]$ and $\deg a_i \leq 2$. We explore the conic bundle structure over the $x$-line to obtain lower and upper bounds for the Mordell--Weil rank of $\cE(k(T))$.
\end{abstract}

\section{Introduction}

Let $k$ be a number field, and $\cE$ a curve over the function field $k(T)$ given by an equation of the form

\begin{equation}\label{equation1} 
    y^2 = a_3(T)x^3 + a_2(T)x^2 + a_1(T)x + a_0(t),
\end{equation}

where $a_i(T)$ are polynomials in $k[T]$ of degree at most 2. We further assume that
\begin{equation*}
    \Deltaell(T) = a_3^2(-27a_0^2 a_3^2 + 18a_0a_1a_2a_3 + a_1^2a_2^2 - 4a_0a_2^3 - 4a_1^3a_3)
\end{equation*}
is not identically equal to $0$, and $a_i(T)$ are not all multiple of the same square ${(T-c)^2}$. With these conditions, $\cE$ defines an elliptic curve over $k(T)$. Curves in this form have been studied in \cite{ALRM2007}, \cite{kollar-mella}, \cite{battistoni2021ranks}. In this paper, we are interested in the rank $r_k$ of $\cE(k(T))$. By the condition on the degrees of the coefficients, we can rewrite Equation \ref{equation1} as

\begin{equation}\label{equation2}
    y^2 = A(x)T^2 + B(x)T + C(x).
\end{equation}

\sloppy Our goal is to understand how the conic bundle structure described in Equation \ref{equation2} influences the rank of $\cE$ over $k(T)$. Firstly, we notice that each root $\theta$ of $\Deltaconic(x) \colonequals {B(x)^2 - 4A(x)C(x)}$ induces a $\overline{k}(T)$-point $P_{\theta} \in \cE(\overline{k}(T))$. Namely,

\begin{equation}\label{points from CB}
    P_{\theta} = \begin{cases}
        \Bigl(\theta, \sqrt{A(\theta)} \bigl(T - \frac{B(\theta)}{2A(\theta)}\bigr)\Bigr) & \text{when $A(\theta) \neq 0$,}\\
        \hfil\Bigl(\theta, \sqrt{C(\theta)}\Bigr) & \text{when $A(\theta) = 0$.}
    \end{cases}
\end{equation}

In some cases, it is known that the number of points induced by roots of $\Deltaconic(x)$ which are defined over $k$ determines $r_k$ exactly (see \cite[Theorem 2.1]{ALRM2007}). This is not true in general -- indeed, there are cases where $r_k$ is strictly smaller than the number of roots of $\Deltaconic$ (see \cite[Theorem $A_2$]{shioda91}).

\subsection{Our contribution}

We approach the problem through the underlying geometry determined by Equations \ref{equation1} and \ref{equation2}. From Equation \ref{equation1}, we obtain the Kodaira--Néron model of $\cE$, that is, a rational elliptic surface $\pi \colon X \to \PP^1$ with $\cE$ as its generic fiber. On the other hand, Equation \ref{equation2} displays a conic bundle structure $\varphi \colon X \to \PP^1$. We show that the $\overline{k}(T)$-points of $\cE$ induced by roots of $\Deltaconic$ correspond to $(-1)$-components of reducible fibers of $\varphi$. In particular, every reducible conic fiber is of type $D_n$, for $n \geq 3$, or of type $A_n$, for $n \geq 2$ (see \cite[Theorem 4.2]{renato}); a conic fiber of type $D_n$ has a single $(-1)$-component corresponding to a point of $2$-torsion in $\cE(\overline{k}(T))$, and a conic fiber of type $A_n$ has a pair of $(-1)$-components such that the corresponding points are additive inverses in $\cE(\overline{k}(T))$.

Firstly, we work over the algebraic closure $\overline{k}$. In this context, we show in Propositions \ref{determining_conic_fibers} and \ref{fiber_at_infinity} that we can determine the type of each reducible fiber of $\varphi \colon X \to \PP^1$ from Equation \ref{equation2}. This allows us to calculate the number $\delta$ of fibers of type $A_n$. We can also find the rank $r$ of $\cE(\overline{k}(T))$ by combining Tate's Algorithm and the Shioda--Tate formula. We call the difference $\delta - r$ the \textit{defect} of $\cE$ and denote it by $\Df(\cE)$.

Over $k$, we define the number $\delta_k$ according to the action of $\Galk$ on the components of the fibers of type $A_n$ of $\varphi$. With these two constants in hand, we give the following bounds for the Mordell--Weil rank over $k$.

{
\renewcommand{\thethm}{\ref{theorem_rank_k}}
\begin{thm}
  Let $r_k$ be the rank of $\cE(k(T))$. Then, $\delta_k \geq r_k \geq \delta_k - \Df(\cE)$.
\end{thm}
\addtocounter{thm}{-1}
}

We explore in which cases we can determine $r_k$ exactly. In Theorem \ref{theorem: Determine Df}, we show that for any given curve $\cE$, the type of the fiber $G_{\infty}$ of $\varphi$ at infinity and the type of the fibers of $\pi$ which have components in common with $G_{\infty}$ are enough to determine the defect $\Df(\cE)$. We apply this to find families of curves for which every member $\cE$ has $\Df(\cE) = 0$, and consequently $r_k = \delta_k$. For some families with $\Df(\cE) = 1$, we are able to use the Galois action on the points $P_{\theta}$ to determine whether $r_k = \delta_k$ or $r_k = \delta_k - 1$.

With these strategies, we are able to determine the rank $r_k$ for many subfamilies in the larger family of curves given by Equation \ref{equation1}. In Theorem \ref{theorem: generic is 2I2}, we prove that for a general member of this family, $r_k = \delta_k$. Specifically, we show that $\Df(\cE) = 0$ outside of Zariski closed set determined by conditions on the coefficients of the polynomials $a_i(T)$.

We also apply Theorems \ref{theorem_rank_k} and \ref{theorem: Determine Df} to recover previous results in \cite{ALRM2007} and \cite{battistoni2021ranks}. The recovered results are stated in Section \ref{section:survey}. At last, we exhibit formulas for the ranks $r_k$ in families not covered by \cite{ALRM2007} and \cite{battistoni2021ranks} (see Example \ref{example1} and Proposition \ref{prop: A=0, deg(B)=3}). 

\subsection{Organization of the paper}

Section \ref{section:survey} surveys Nagao's conjecture and its applications to determine the Mordell--Weil rank of elliptic curves as in Equation \ref{equation1}.

Section \ref{section:preliminaries} is a brief overview of some pre-requisites for the main part of the paper. Namely, we present the basic definitions of elliptic surfaces, conic bundles, and some of the main results in their study.

Section \ref{section_CB_and_RES} deals with a general rational surface $X$ with a relatively minimal elliptic fibration $\pi \colon X \to \PP^1$ and a conic bundle $\varphi \colon X \to \PP^1$. In \ref{section:generalities}, we adapt results from standard conic bundles to obtain a description of the Néron--Severi group $\NS(\overline{X})$ and the canonical divisor $K_{X}$ (see Proposition \ref{prop_basis_canonical}). In \ref{subsection:rank delta}, we compare the number $\delta$ of fibers of type $A_n$ in $\varphi$ and the rank $r$ of the generic fiber of $\pi$.

In Section \ref{section:CBs_Weierstrass} we apply the results of the previous section to surfaces $X$ given by Equation \ref{equation1} and Equation \ref{equation2}. In \ref{subsection: determining fiber types}, we determine the types of the singular fibers of the conic bundle $\varphi$ from Equation \ref{equation2} (see Proposition \ref{determining_conic_fibers}). In \ref{section:defect_rank}, we define the defect of $X$, and prove that the points $P_{\theta}$ induced by the conic bundle $\varphi$ generate a finite index subgroup of $\cE(\overline{k}(T))$ (see Theorem \ref{theorem:kernel_equals_defect} and Corollary \ref{coro:finite index subgroup}). In \ref{section:rank k bounds} we define the number $\delta_k$ in terms of the action of $\Galk$ on the fibers of $\varphi$ and prove Theorem \ref{theorem_rank_k}).

Section \ref{section: computations of the rank} deals with the application of Theorem \ref{theorem_rank_k} to determine the rank $r_k$ of families of curves given by Equation \ref{equation1}. In 
\ref{section: computation of defect}, we determine $\Df(\cE)$ from the fiber configuration of $\pi$ and $\varphi$ (see Theorem \ref{theorem: Determine Df}). In \ref{section: Df = 0}, we explore cases in which $\Df(\cE) = 0$, and provide families for which $r_k = \delta_k$. In \ref{section: Df > 0} we explore cases with $\Df(\cE) > 0$, and provide families with $\Df(\cE) = 1$ for which we can determine if $r_k = \delta_k$ or $r_k = \delta_k - 1$ depending on the coefficients of $\cE$ in Equation \ref{equation2}. The computations in this section were done with the aid of the computer algebra system Magma \cite{magma}.

\begin{acknowledge}
    The author thanks Cecília Salgado and Jaap Top for many producitve conversations regarding this topic, and specifically to the first for the suggestion of the problem. The author was financially supported by a scholarship from Capes and the University of Groningen. 
\end{acknowledge}

\section{Survey of related results}\label{section:survey}

In this section, we state Nagao's conjecture and give a brief exposition on subsequent theorems and applications. We follow the original exposition of the conjecture in \cite{nagao1997}, so we work over the $\QQ$.

Let $\cE$ be an elliptic curve over $\QQ(T)$, given by an equation in Weierstrass form:

\begin{equation*}
    \cE: \ y^2 + a_1(T)xy + a_3(T) y = x^3 + a_2(T)x + a_4(T)x + a_6(T),
\end{equation*}

with coefficients $a_i(T) \in \ZZ[T]$. For each $t \in \QQ$, the specialization map $T \mapsto t$ yields an elliptic curve $\cE_t$ over $\QQ$. For each prime $p \in \ZZ_{>0}$ of good reduction, we consider $a_t(p)$ the trace of the Frobenius at $p$ on $\cE_t$, given by $p + 1 - N_t(p)$, where $N_t(p)$ is the number of $\FF_p$ points of $\cE_t$ (we say that $a_t(p) = 0$ if $p \mid \Delta(t)$). Further, consider the average over fibers

\begin{equation*}
 A_{\cE}(p) := \frac{1}{p}\sum_{t=0}^{p-1} a_t(p).
\end{equation*}

In 1997, based on several explicit calculations for the Mordell--Weil rank of $\cE(\QQ(T))$ on nontrivial families, Nagao conjectured a limit formula relating $\rank \cE(\QQ(T))$ to the values of $A_{\cE}(p)$ (see \cite[Question (2)]{nagao1997}).

\begin{conjecture}
Let $\cE$, $A_{\cE}(p)$ be as defined above, then 
\begin{equation*}
    \lim_{X \rightarrow \infty} \frac{1}{X} \sum_{p \leq X} -A_{\cE}(p) \log p = \rank \cE(\QQ(T)).
\end{equation*}
\end{conjecture}

In the following year, Rosen and Silverman published a proof that Tate's conjecture implies Nagao's conjecture, settling it in particular for rational elliptic surfaces (see \cite[Theorem 1.3]{rosen-silverman1998}).

\begin{thm}[Rosen, Silverman]
Nagao's conjecture holds for rational elliptic surfaces.
\end{thm}

In \cite{ALRM2007}, Arms, Lozano-Robledo and Miller apply Rosen and Silverman's result to elliptic curves over $\QQ(T)$ given by an equation of the form

\begin{align}\label{alrm_equation}
    y^2 &= x^3T^2 + 2g(x)T -h(x), \text{ where}\\
    \nonumber g(x) &= x^3 + ax^2 + bx + c, \, c \neq 0;\\
    \nonumber h(x) &= (A-1)x^3 + Bx^2 + Cx + D.
\end{align}

Applying the coordinate change $x \mapsto \tfrac{x}{T^2+2T-A+1}$ to Equation \ref{alrm_equation}, we check that it indeed corresponds to a rational elliptic surface (see \cite[Chapter 5.13]{schutt-shioda}). The theory of quadratic Legendre sums is used to prove the following.

\begin{thm}\label{ALRM}
    For infinitely many integers $a,b,c,A,B,C,D$, the polynomial $D_T(x) = g(x)^2 + x^3h(x)$ has 6 distinct, nonzero roots which are perfect squares, and the curve $\cE$ given by Equation \ref{alrm_equation} is an elliptic curve over $\QQ(T)$ with $\rank \cE(\QQ(T)) = 6$.
\end{thm}

\begin{proof}
    See \cite[Theorem 2.1 and Remark 2.2]{ALRM2007}.
\end{proof}

This was later generalized to any number field in \cite[Theorem 1.1]{miller2015}.

A similar strategy was employed by Battistoni, Bettin and Delaunay in \cite{battistoni2021ranks} to obtain rank formulas for elliptic curves over $\QQ(T)$ of the form

\begin{equation}\label{equation_BBD}
    y^2 = A(x)T^2 + B(x)T + C(x),
\end{equation}
where $\deg A(x), \deg B(x) \leq 2$, at least one of $A(x)$, $B(x)$ is not identically zero, and $C(x)$ is monic and of degree 3. They obtain four distinct formulas for $\rank \cE(\QQ(T))$, depending on if $A(x) = 0$, $A(x) = \mu \in \QQ^{\times}$, $\deg(A) = 1$ or $\deg(A) = 2$, which we omit for simplicity (see \cite[Theorem 1]{battistoni2021ranks}).

When $A(x) = 0$, the formula is simplified as follows.

\begin{thm}[Battistoni, Bettin, Delaunay]\label{rank A=0}
    Let $\cE$ be an elliptic curve over $\QQ(T)$ given by an equation of the form \ref{equation_BBD}, and assume $A(x) = 0$. Then, the rank of $\cE$ is as follows.

    \begin{equation*}
        \rank \cE(\QQ(T)) = \#\{[\theta]: B(\theta) = 0, C(\theta) \in \QQ(\theta)^2\setminus\{0\}\},
    \end{equation*}
    where $[\theta]$ denotes the orbit of $\theta$ by the action of the Galois group $\Gal(\overline{\QQ}/\QQ)$.
\end{thm}

\begin{proof}
    See \cite[Theorem 2]{battistoni2021ranks}.
\end{proof}

When $A(x) = \mu \in \QQ^{\times}$, the formula is also simplified.

\begin{thm}[Battistoni, Bettin, Delaunay]\label{rank deg(A) = 0}
Let $\cE$ be an elliptic curve over $\QQ(T)$ given by an equation of the form \ref{equation_BBD}, and assume $A(x) = \mu \in \QQ^{\times}$. Then, the rank of $\cE$ is as follows.
    \begin{equation*}
        \rank \cE(\QQ(T)) = \begin{cases}
            \#\{[\theta]: B^2(\theta) - 4\mu C(\theta) = 0)\} - 1 & \text{if $\mu \in \QQ^2$,}\\
            \#\{[\theta]: B^2(\theta) - 4\mu C(\theta) = 0), \mu \in \QQ(\theta)^2\setminus\{0\}\} & \text{if $\mu \in \QQ \setminus \QQ^2$}.
            \end{cases}
    \end{equation*}
\end{thm}

\begin{proof}
    See \cite[Theorem 1, Section 3.1]{battistoni2021ranks}.
\end{proof}

When $\deg(A) = 1$ or $2$, the formula provides the following upper bound for the rank.

\begin{thm}[Battistoni, Bettin, Delaunay]\label{rank deg(A) > 0}
    Let $\cE$ be an elliptic curve over $\QQ(T)$ given by an equation of the form \ref{equation_BBD}. Then, we have the following upper bound for the rank of $\cE$.
    \begin{equation*}
       \rank \cE(\QQ(T)) \leq \#\{[\theta]: B^2(\theta) - 4A(\theta)C(\theta) = 0\},
    \end{equation*}
    where $[\theta]$ denotes the orbit of $\theta$ by the action of the Galois group $\Gal(\overline{\QQ}/\QQ)$.
\end{thm}

\begin{proof}
    See \cite[Proposition 14]{battistoni2021ranks}.
\end{proof}

\section{Preliminaries}\label{section:preliminaries}

Throughout this section, let $k$ be a number field and $\overline{k}$ a fixed algebraic closure. For $X$ an algebraic surface defined over $k$, its model over $\overline{k}$ is $\overline{X} \colonequals X \times_k \overline{k}$.

\subsection{Elliptic surfaces}

Let $\cE$ be an elliptic curve over the function field $k(\PP^1) = k(T)$. The groups $\cE(\overline{k}(T))$ and $\cE(k(T))$ are finitely generated by a consequence of the Lang--Néron Theorem (see \cite{LangNeron}). Let $r$ and $r_{k}$ denote their respective ranks. The curve $\cE$ admits a  Weierstrass equation:

\begin{equation}\label{generic_fiber_equation}
    y^2 + a_1(T) y + a_3(T) xy = x^3 + a_2(T) x^2 + a_4(T) x + a_6(T).
\end{equation}

We homogenize Equation \ref{generic_fiber_equation} twice: once with respect to the coordinates $x,y$, and then with respect to the parameter $T$. This define a surface $W \subset \PP^2_{[x{:}y{:}z]} \times \PP^1_{[S{:}T]}$, with a natural map $W \rightarrow \PP^1$. After blowing up points of $W$ to resolve singularities, we obtain an elliptic surface $\pi \colon X \to \PP^1$ defined over $k$.

\begin{dfn}
 An \textit{elliptic surface} over $k$ is defined as a smooth algebraic surface $X$ endowed with a fibration $\pi \colon X \to C$ to a smooth curve $C$, where $X, C, \pi$ are all defined over $k$ and

\begin{itemize}
    \item[i)] all but finitely many fibers of $\pi$ are smooth irreducible genus 1 curves;
    \item[ii)] there is at least one singular fiber;
    \item[iii)] there is at least one section $\sigma_0:C \rightarrow X$ defined over $k$, (referred to as the \textit{zero-section}).
\end{itemize}

\begin{assumption}
    We restrict ourselves to the case $C \cong \PP^1$, as the elliptic surfaces we are interested in always have $\PP^1$ as a base curve.   
\end{assumption}

If no fiber components are $(-1)$-curves, we say that the elliptic surface is \textit{relatively minimal}. In what follows, every elliptic surface is assumed to be relatively minimal.
\end{dfn}

The elliptic surface $\pi: X \rightarrow \PP^1$ obtained from $\cE$ is called the \textit{Kodaira-Néron model} of $\cE$, and it is unique up to isomorphism. Furthermore, the generic fiber of $\pi$ is isomorphic to $\cE$ (see \cite[Chapter IV, Theorem 6.1]{silverman2}).

\begin{thm}
There is a bijection between the set of sections $\sigma \colon \PP^1 \rightarrow \overline{X}$ and the set of $\overline{k}(T)$-points of $\cE$. Furthermore, if $\sigma$ is defined over $k$, then it corresponds to a $k(T)$-point.
\end{thm}

\begin{proof}
    See \cite[Chapter III, Proposition 3.10.(c)]{silverman2}.
\end{proof}

\begin{notation}
We denote the identity of $\cE(\overline{k}(T))$ by $O$. For any $P,Q \in \cE(\overline{k}(T))$, we denote their sum as $P \oplus Q$, the sum of $P$ with itself $n$ times by $[n]P$ and its respective inverse as $[-n]P$.  Let $\sigma \colon \PP^1 \to X$ be a section corresponding to a point $P$. Then, we also refer to the rational curve $\sigma(\PP^1) \subset X$ as a section, and denote it by $(P)$. We assume the zero-section of $\pi$ is equal to $(O)$. 
\end{notation}

We now turn to the fibers of $\pi \colon X \to \PP^1$. For each $t \in \PP^1_{\overline{k}}$, let $F_t$ denote the fiber $\pi^{-1}(t)$. By definition, at least one fiber is singular, and in total the number of singular fibers is finite. The singular fibers in an elliptic fibration have been classified in different types by Kodaira and Néron (see \cite{kodaira1963compact}, \cite{Neron1964}), and Tate's Algorithm describes how to determine the type of a particular fiber $F_t$ from Equation \ref{generic_fiber_equation} (see \cite{tate1975algorithm}, Sections 7 and 8). The Galois group $\Galk$ acts on the set of irreducible components of singular fibers of $\pi$. Furthermore, $\Galk$ permutes fibers $F_t$ not defined over $k$, preserving their Kodaira types.

\subsection{The Néron--Severi lattice}

Let $\NS(\overline{X})$ be the Néron--Severi group of $\overline{X}$ and denote its rank by $\rho(\overline{X})$.

\begin{prop}
Let $X$ be an elliptic surface. Then $\NS(\overline{X})$ has a lattice strucuture with the intersection pairing $(\,\cdot\,)$, which we call the Néron--Severi lattice. In other words, the intersection pairing is torsion free.
\end{prop}

\begin{proof}
See \cite[Theorem 1.2]{shioda90}.
\end{proof}

\begin{dfn}
    Let $\pi \colon X \to \PP^1$ be an elliptic surface over $k$ and $C$ be a prime divisor of $\overline{X}$. We say that $C$ is \textit{vertical} (with respect to $\pi$) if $\pi(C) = v \in \PP^1$. We say that $D \in \NS(\overline{X})$ is vertical if $D$ is represented by a sum of prime divisors which are all vertical. We define $\Triv(\overline{X})$ as the sublattice of $\NS(\overline{X})$ generated by the vertical divisors with respect to $\pi$ and the zero-section.
\end{dfn}

\begin{thm}\label{shioda-tate-isom}
Let $\pi \colon X \to \PP^1$ be an elliptic surface over $k$ with generic fiber $\cE$. Then, there is an isomorphism
\begin{equation*}
    \frac{\NS(\overline{X})}{\Triv(\overline{X})} \cong \cE(\overline{k}(C)).
\end{equation*}
\end{thm}

\begin{proof}
    See \cite[Theorem 1.3]{shioda90}.
\end{proof}

The lattice $\Triv(\overline{X})$ is generated by $O$, the fiber class $F$, and every irreducible fiber component not intersecting $O$. Therefore, $\rank(\Triv(\overline{X})) = 2 + \sum_{v \in \PP^1}(m_v -1)$, where $m_v$ is the number of components of the fiber $F_v$. As a corollary, we have the following formula.

\begin{corollary}[The Shioda--Tate formula]
\label{shiodatate}
Let $r$ be the rank of $\cE(\overline{k}(T))$ and $m_v$ the number of components of the fiber $F_v$. We have
\begin{equation*}
    \rho(X) = r + 2 + \sum (m_v -1).
\end{equation*}
\end{corollary}

 The following proposition shows a connection between linear relations of $\overline{k}(T)$-points in $\cE(\overline{k}(T))$ and the vertical divisors of $\pi \colon X \to \PP^1$.

\begin{prop}\label{silverman_vertical_divisors}
    Let $P_1, \ldots, P_m$ be $\overline{k}(T)$-points on the generic fiber $\cE$ of an elliptic surface $\pi \colon X \to \PP^1$, and let $n_1, \ldots, n_m \in \ZZ$ be integers such that
    \begin{equation*}
        [n_1]P_1 \oplus \cdots \oplus [n_m] P_m = O.    
    \end{equation*}
    Then for $n = n_1 + \cdots + n_m$, the divisor $n_1(P_1) + \cdots + n_m(P_m) - n(O) \in \NS(\overline{X})$ is vertical.
\end{prop}

\begin{proof}
    See \cite[Chapter III, Proposition 9.2]{silverman2}.
\end{proof}

\subsection{Rational elliptic surfaces}

\begin{dfn}
    An elliptic surface $\pi \colon X \to \PP^1$ over $k$ is called a \textit{rational elliptic surface} if it is \textit{geometrically rational}, that is, if $\overline{X}$ is birational to $\PP^2_{\overline{k}}$.
\end{dfn}

Given two plane cubics $\cF$ and $\cG$, one can construct a rational elliptic surface by resolving the indeterminacy of the map $\varphi \colon \PP^2 \dashrightarrow \PP^1$, where $\varphi(P) = [\cF(P){:}\cG(P)]$. We do this by blowing up the nine scheme-theoretic points in the intersection $\cF \cap \cG$, thus obtaining a surface $X$ with an genus 1 fibration $\pi \colon X \rightarrow \PP^1$. The exceptional divisor of the last blow-up is a section of $\pi$, this we obtain a rational elliptic surface.

\begin{thm}\label{res_cubic_pencil}
Let $X$ be a rational elliptic surface defined over $\overline{k}$. Then, $X$ is birational (over $\overline{k}$) to the blow-up of $\PP^2$ at the base points of a pencil of cubics.
\end{thm}

\begin{proof}
See \cite[Lemma IV.1.2]{miranda}.
\end{proof}

As a consequence of this construction, we obtain some geometric properties of any rational elliptic surface.

\begin{prop}\label{results_RES}
Let $\pi \colon X \to \PP^1$ be a rational elliptic surface defined over $\overline{k}$. Further, let $e(\overline{X})$ denote the Euler number of $\overline{X}$ and $K_X$ the canonical divisor of $X$. The following hold.
\begin{itemize}
    \item[i)] $\rho(\overline{X}) = 10$.
    \item[ii)] $e(\overline{X}) = 12$.
    \item[iii)] $K_X = -F$.
    \item[iv)] A rational curve $C \subset X$ is a section of $\pi$ if and only if $C^2 = -1$.
    \item[v)] A rational curve $C \subset X$ is a fiber component of $\pi$ if and only if $C^2 = -2$.
\end{itemize}
\end{prop}

\begin{proof}
See \cite[Proposition 7.5]{schutt-shioda} for items (i) and (ii), \cite[Proposition 5.28]{schutt-shioda} for item (iii). Items (iv) and (v) are a consequence of the Adjunction Formula (see \cite[Theorem I.15]{beauville}). 
\end{proof}

Notice that (i) in Proposition \ref{results_RES} simplifies the Shioda--Tate formula (Corollary \ref{shiodatate}) when $X$ is a rational elliptic surface, as stated by the following corollary.

\begin{corollary}\label{rational_shioda-tate}
    Let $\pi \colon X \to \PP^1$ be a rational elliptic surface and $r$ the rank of its generic fiber $\cE$ over $\overline{k}(T)$. Then,
    \begin{equation*}
        r = 8 - \sum_{v \in \PP^1} (m_v -1).
    \end{equation*}
\end{corollary}

\subsection{Conic bundles}

\begin{dfn}
Let $X$ be a surface over $k$. A \textit{conic bundle} with rational base of $X$ over $k$ is a surjective morphism $\varphi \colon X \to \PP^1$ defined over $k$, such that all but finitely many fibers are irreducible curves of genus $0$. Furthermore, if every reducible fiber is given by two distinct rational curves intersecting in a single point, we say that $\varphi$ is \textit{standard}.
\end{dfn}

As we are interested solely on rational surfaces, all conic bundles have rational base. We will refer to them in what follows simply as conic bundles, omitting the rational base.

\begin{thm}\label{standard_conic_bundles}
Let $X$ be a rational surface and $\varphi \colon X \to \PP^1$ a standard conic bundle of degree $d = K_X^2$. Then, the following hold. 

\begin{itemize}
\item[i)] There are $r' = 8-d$ reducible fibers of $\varphi$, all of which are composed of two concurrent exceptional curves.
\item[ii)] There is a free basis of $\NS(\overline{X})$ given by $\langle G, H, \ell_1, {.}{.}{.}, \ell_{r'}\rangle$, where $G$ is the fiber class of $\varphi$, $H$ is a section of $\varphi$ and $\ell_1,...,\ell_{r'}$ are the components of the reducible fibers of $\varphi$ not intersecting $H$.
\item[iii)]The canonical divisor of $X$ is given by
\begin{center}
    $K_X = -2H + (H^2-2)G + \sum_{i=1}^{r'} \ell_i$.
\end{center}
\end{itemize}
\end{thm}

\begin{proof}
    See \cite[Proposition 0.4]{tsfasman}.
\end{proof}

Let $X$ be a rational elliptic surface. Conic bundles on rational elliptic surfaces were studied in \cite{AGL} and \cite{renato}. The following result gives a classification of the singular fibers of a more general, not necessarily standard, conic bundle over $X$.

\begin{thm}\label{CBfibers}
Let $\pi \colon X \to \PP^1$ be a rational elliptic surface with fiber class $F$ and $\varphi \colon X \to \PP^1$ a conic bundle with fiber class $G$. Then, $G^2 = 0$ and $F \cdot G = 2$. Furthermore, every fiber of $\varphi$ is of one of the types in Table \ref{CBfibers_table}.
\end{thm}

\begin{proof}
See \cite[Theorem 3.8, Theorem 4.2]{renato}.
\end{proof}

\begin{table}[h]
    \centering
    \begin{tabular}{|c|c|}
        \hline
        & \\[-1em]
        Type & Intersection Graph\\
        & \\[-1em]
        \hline
        & \\[-1em]
        $0$ & $*$\\
        & \\[-1em]
        \hline
        & \\[-1em]
        $A_2$ & \input{conicfiberA2.tikz}\\
        & \\[-1em]
        \hline
        & \\[-1em]
        $A_n \ (n \geq 3)$ & \input{conicfiberAn.tikz}\\
        & \\[-1em]
        \hline
        & \\[-1em]
        $D_3$ & \input{conicfiberD3.tikz}\\
        & \\[-1em]
        \hline
        & \\[-1em]
        $D_n \ (n \geq 4)$ & \input{conicfiberDn.tikz}\\
        & \\[-1em]
        \hline
    \end{tabular}\\
    \begin{minipage}{0.5\textwidth}
      \begin{itemize}
        \item[$*$]\tiny smooth, irreducible curve of genus zero 
        \item[\input{whitedot.tikz}]\tiny $(-1)$-curve (section of $\pi$)
        \item[\input{blackdot.tikz}]\tiny $(-2)$-curve (component of a reducible fiber of $\pi$)
      \end{itemize}
    \end{minipage}
    \caption{fibers in conic bundles over rational elliptic surfaces}
    \label{CBfibers_table}
\end{table}

\begin{remark}\label{standard_is_A2}
    By the classification in Table \ref{CBfibers_table}, we can define a stardard conic bundle as a conic bundle in which every reducible fiber is of type $A_2$.
\end{remark}

\section{Conic bundles on rational elliptic surfaces}\label{section_CB_and_RES}

\subsection{Generalities on conic bundles}\label{section:generalities}

Let $\pi \colon X \to \PP^1$ be a rational elliptic surface, and $\varphi \colon X \to \PP^1$ a conic bundle. By Theorem \ref{CBfibers}, every reducible fiber of $\varphi$ is of type $A_n$, with $n \geq 2$, or $D_n$, with $n \geq 3$. Let $\delta(\varphi)$ be the number of fibers of type $A_n$, and $\varepsilon(\varphi)$ the number of fibers of type $D_n$. Notice that these numbers depend on the conic bundle $\varphi$; a rational elliptic surface may be endowed with two different conic bundles $\varphi_1$ and $\varphi_2$ such that $\delta(\varphi_1) \neq \delta(\varphi_2)$. Through the rest of this section, we fix one conic bundle $\varphi$ and refer to $\delta(\varphi), \varepsilon(\varphi)$ as $\delta, \varepsilon$, respectively.
Next, we establish further notation for the fibers of $\varphi$.

\begin{notation}\label{notation_CB_fibers}
    A fiber $\varphi^{-1}(v)$ is denoted by $G_v$, its number of components by $n_v$ and its class in $\NS X$ by $G$. 

    Denote the fibers of type $A_n$ by $G_{v_1}, \ldots, G_{v_{\delta}}$. We write

\begin{equation}\label{fiber_type_An}
    G_{v_i} = \sum_{j=0}^{n_{v_i}\text{-}1}\alpha_{i,j}\text{.}
\end{equation}

The components in Equation \ref{fiber_type_An} intersect following the graph of fibers of type $A_n$ in Table \ref{CBfibers_table}, with $\alpha_{v_i,0}^2 = \alpha_{v_i,n_{v_i}\text{-}1}^2 = -1$ and $\alpha_{v_i,j}^2 = -2$ for ${j = 1, \ldots, n_{v_i}-2}$.

Denote the fibers of type $D_n$ by $G_{w_1}, \ldots, G_{w_{\varepsilon}}$. We write

\begin{equation}\label{fiber_type_Dn}
    G_{w_i} = \beta_{w_i,0} + \beta_{w_i,1} + 2\sum_{j=2}^{n_{w_i}\text{-}1} \beta_{w_i,j}\text{.}
\end{equation}

Similarly, the components in Equation \ref{fiber_type_Dn} intersect following the graph of fibers of type $D_n$ in Table \ref{CBfibers_table}, with $\beta_{w_i,n_{w_i}\text{-}1}^2 = -1$ and $\beta_{w_i,j}^2 = -2$ for $j = 0, \ldots, n_{w_i} - 2$.
\end{notation}

Let $G_v$ be a reducible fiber of $\varphi$. Then $e(G_v) = n_{v}+1$ independently on the type of $G_v$, since the fiber is composed of $n_v$ rational curves intersecting at $n_{v}-1$ distinct points. We can use this fact to limit the possible configurations of fibers.

\begin{prop}\label{conic_bundle_configuration}
    For $\varphi \colon X \to \PP^1$ a conic bundle over a rational elliptic surface, we have the following formula.
    
    \begin{equation*}
        \sum_{v \in \PP^1} (n_v - 1) = 8.
    \end{equation*}
    
\end{prop}

\begin{proof}
    By \cite[Proposition 5.1.6]{cossec-dolgachev}, the Euler number of $X$ is given by
    \begin{align*}
        e(\overline{X}) = e(G_{\eta})e(\PP^1) + \sum_{v \in \PP^1}(e(G_v) - e(G_{\eta})),
    \end{align*}
    where $G_{\eta}$ is the generic fiber of $\varphi$. By Proposition \ref{results_RES}, $e(\overline{X}) = 12$. Substituting $e(G_{\eta}) = e(\PP^1) = 2$ and $e(G_v) = n_v + 1$, we obtain the result.
\end{proof}

Let $H \subset X$ be a section of $\varphi$. Then, $H \cdot G_v = 1$, so $H$ intersects a single simple component of $G_v$. For the fibers $G_{w_i}$ of type $D_n$, $H$ can only intersect $\beta_{w_i,0}$ or $\beta_{w_i,1}$, and we can assume without loss of generality that it intersects $\beta_{w_i,0}$. On the other hand, for the fibers $G_{v_i}$ of type $A_n$, $H$ can intersect any component. Let $k_i$ be the number such that $H$ intersects $\alpha_{v_i,k_i}$. We can assume without loss of generality that $0 \leq k_i \leq n_{v_i} - 2$.

\begin{prop}\label{blow-down_to_standard}
    Let $X$ be a rational elliptic surface and $\varphi \colon X \to \PP^1$ a conic bundle on $X$ with a section $H \subset X$. Then, there is a contraction $\eta \colon X \to X_0$ such that
    \begin{itemize}
        \item [i)] $H$ does not intersect any of the curves contracted by $\eta$;
        \item[ii)] there is a standard conic bundle $\varphi_0 \colon X_0 \to \PP^1$ such that $\varphi = \varphi_0 \circ \eta$.
    \end{itemize}
    In other words, $\eta(H)$ is a section of $\varphi_0 \colon X_0 \to \PP^1$.
\end{prop}

\begin{proof}
    Let $E$ be a $(-1)$-component of a fiber of $\varphi$. Then, the pushforward of $G$ by the blow-down of $E$ induces a conic bundle commuting with the blow-down map.
    
    Let $G_{v_i}$ be a fiber of type $A_n$ such that $n_{v_i} \geq 3$. If we blow-down one of the $(-1)$-components of $G_{v_i}$, the $(-2)$-component intersecting it becomes a $(-1)$-component on the contracted surface. Thus, we can repeat this process successively. We need to contract $n_{v_i} - 2$ components that do not intersect the section $H$. Since we assume that $H$ intersects $\alpha_{v_i,k_i}$, we can do this by blowing-down a chain of $k_i$ components starting with $\alpha_{v_i,0}$ and a chain of $n_{v_i}-k_{i}-2$ components starting with $\alpha_{v_i,n_{v_i}\text{-}1}$. This iterative process yields a fiber of type $A_2$ (see Figure \ref{blowdown_to_standard_An}).

    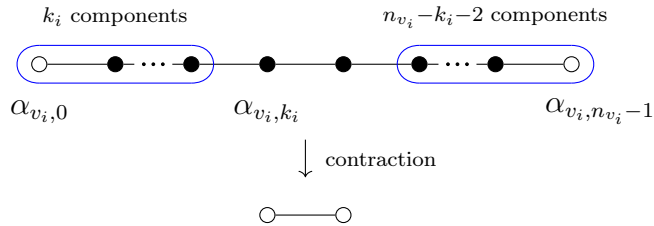
\begin{figure}[h]
        \centering
        \input{blowdown_to_standard_An.tikz}
        \caption{Blowing-down a fiber of type $A_n$ to a fiber of type $A_2$}
        \label{blowdown_to_standard_An}
    \end{figure}

    Let $G_v$ be a fiber of type $D_n$, with $n \geq 3$. Since we assume $H$ intersects the component $\beta_{w_i,0}$, we can similarly blow-down a chain of $n_{w_i}-2$ components starting with $\beta_{w_i,n_{w_i}\text{-}1}$, reaching a fiber of type $A_2$ (see Figure \ref{blowdown_to_standard_Dn}).

    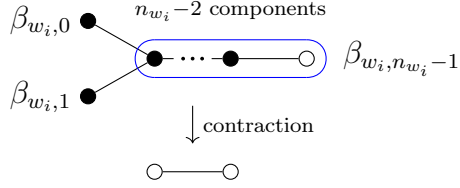
\begin{figure}[h]
        \centering
        \input{blowdown_to_standard_Dn.tikz}
        \caption{Blowing-down a fiber of type $D_n$ to a fiber of type $A_2$}
        \label{blowdown_to_standard_Dn}
    \end{figure}

    Applying these blow-downs to all reducible fibers of $\varphi$, we reach a conic bundle in which all reducible fibers are of type $A_2$. By Remark \ref{standard_is_A2}, this is a standard conic bundle.
\end{proof}

In what follows, we use Proposition \ref{blow-down_to_standard} to generalize Theorem \ref{standard_conic_bundles} to general conic bundles over rational elliptic surfaces.

\begin{prop}\label{prop_basis_canonical}
    Let $\varphi \colon X \to \PP^1$ be a conic bundle over a rational elliptic surface and $H$ a section of $\varphi$. The following hold.
    \begin{itemize}
        \item[i)] There is a free basis of $\NS \overline{X}$ given by
        \begin{align*}
            \cB = \langle &G, \alpha_{v_1,1}, {.}{.}{.} \;, \alpha_{v_1,n_{v_1}\text{-}1}, {.}{.}{.} \;, \alpha_{v_{\delta},1}, {.}{.}{.} \;, \alpha_{v_{\delta},n_{v_{\delta}}\text{-}1},\\
            & H, \beta_{w_{1},1}, {.}{.}{.}\, , \beta_{w_1,n_{w_{1}}\text{-}1}, {.}{.}{.} \,, \beta_{w_{\varepsilon},1}, {.}{.}{.}, \beta_{w_{\varepsilon},n_{w_{\varepsilon}}\text{-}1}\rangle.
        \end{align*}
    \item [ii)] The canonical divisor of $X$ is given by
        \begin{align*}
            K_X = &-2H + \bigl(\sum_{i=1}^{\delta}k_i + H^2{-}2\bigr)G + \sum_{i=1}^{\delta} \bigl(\sum_{j=1}^{n_{v_i}\text{-}1}(|k_i{-}j|-k_i)\alpha_{v_i,j}\bigr) + \sum_{i=1}^{\varepsilon}\bigl(\sum_{j=1}^{n_{w_i} \text{-}1}j\beta_{w_i,j}\bigr).
        \end{align*}
        \end{itemize}
\end{prop}

\begin{proof}
    Let $\eta \colon X \to X_0$ be the contraction to a standard conic bundle in Proposition \ref{blow-down_to_standard}. For each reducible fiber with $n$ components, $n-2$ components are contracted. Thus there are $\sum_{i=1}^{\delta}(n_{v_i}{-}2) + \sum_{i=1}^{\varepsilon}(n_{w_i}{-}2)$ individual contractions. We can decompose $\eta$ as

    \begin{center}
        $\eta = \eta_{v_1,n_{v_1}\text{-}2} \circ \cdots \circ \eta_{v_1,1} \circ \cdots \circ \eta_{v_{\delta},1} \circ \eta_{w_1,n_{\varepsilon}\text{-}2} \circ \cdots \circ \eta_{w_{\varepsilon},1}$,
    \end{center}
    where $\eta_{v_i,j}$ (respectively $\eta_{w_i,j}$) is the $j$-th contraction of a component of $G_{v_i}$ (respectively $G_{w_i}$). These maps correspond to blow-ups with the contracted component as the exceptional divisor. Theorem \ref{standard_conic_bundles} gives us a free basis of $\NS X_0$ and the canonical divisor $K_{X_0}$. In what follows, we apply basic properties of blow-ups (see \cite[Proposition II.3]{beauville}) to each of the maps $\eta_{v_i,j}$ and $\eta_{w_i,j}$.\\
    
    i) The image of $H$ by $\eta$ is a section of the conic bundle $\varphi_0 \colon X_0 \to \PP^1$. Indeed, by Proposition \ref{blow-down_to_standard}, the contracted divisors do not intersect $H$. To simplify notation, we also refer to this section as $H$. For each fiber $G_{v_i}$, all of its components are contracted by $\eta$, except for $\alpha_{v_i,k_i}$ and $\alpha_{v_i,k_i+1}$. Similarly, for each fiber $G_{w_i}$, all components are contracted except for $\beta_{w_i,0}$ and $\beta_{w_i,1}$. We also refer to the image of these components in $X_0$ by the same notation, and to the fiber class of $\varphi_0$ by $G$. Since $H$ intersects the components $\alpha_{v_i,k_i}$ and $\beta_{w_i,0}$, we know by Theorem \ref{standard_conic_bundles} that $\NS \overline{X_0}$ is generated by the free basis
    \begin{center}
         $\langle G, H, \alpha_{v_1,k_i+1}, {.}{.}{.} \;, \alpha_{v_\delta,k_{\delta}+1}, \beta_{w_{1},1}, {.}{.}{.}\, , \beta_{w_{\varepsilon},1}\rangle$.
    \end{center}

    By \cite[Proposition II.3.(iii)]{beauville}, $\NS \overline{X}$ is generated by by the pullback of the basis of $NS \overline{X_0}$ and the exceptional divisors the blow-ups $\eta_{v_i,j},\eta_{w_i,j}$. When $k_i = 0$ for all $i = 1, \ldots,\delta$, this is equal to $\cB$ and we are done. Otherwise, we write $\alpha_{v_i,0}$ in terms of the basis $\cB$ as $G - \alpha_{v_i,1} - \ldots - \alpha_{v_i,n_{v_i}\text{-}1}$, so $\cB$ generates $\NS \overline{X}$.\\

    ii) By Theorem \ref{standard_conic_bundles}, the canonical divisor of $X_0$ is given by
    \begin{equation*}
        K_{X_0} = -2H + (H^2-2)G + \sum_{i=1}^{\delta}\alpha_{v_i,k_i+1} + \sum_{i=1}^{\varepsilon}\beta_{w_i,1}\text{.}
    \end{equation*}
    
    We obtain $K_X$ by applying  of \cite[Proposition II.3.(iv)]{beauville} successively for each individual blow-up $\eta_{v_i,j}, \eta_{w_i,j}$. Firstly, notice that since the section $H$ and a general fiber $G$ do not intersect any of the exceptional divisors, their pullbacks by any $\eta_{v_i,j}, \eta_{w_i,j}$ are given by only their strict transforms. Therefore, we focus on calculating the pullbacks on $\alpha_{v_i,k_i+1}$ and $\beta_{w_i,1}$, as well as the exceptional divisors introduced by each blow-up. We can do this fiber by fiber.
    
    We start with a fiber $\varphi_0^{-1}(w_i)$. The map, $\eta_{w_i,1}$ is centered at a point of $\beta_{w_i,1}$, and its exceptional divisor corresponds to $\beta_{w_i,2}$. Thus, we calculate
    \begin{equation*}
        \eta_{w_i,1}^*(\beta_{w_i,1}) + \beta_{w_i,2} = \beta_{w_i,1} + 2\beta_{w_i,2}\text{.}
    \end{equation*}
    Subsequently, the $j$-th blow-up $\eta_{w_i,j}$ is centered at the component $\beta_{w_i,j}$. Applying this for all $j$ up to $n_{w_i}-2$, we conclude the part of $K_X$ supported in $G_{w_i}$ is equal to
    \begin{equation*}
        \sum_{j=1}^{n_{w_i}\text{-}1}j\beta_{w_i,j}. 
    \end{equation*}
    
    For a fiber $\varphi_0^{-1}(v_i)$, the blow-up $\eta_{v_i,1}$ is centered at $\alpha_{v_i,k_i{+}1}$ with exceptional divisor corresponding to $\alpha_{v_i,k_i{+}2}$. We calculate
    \begin{equation*}
        \eta_{v_i,1}^*(\alpha_{v_i,k_i{+}1}) + \alpha_{v_i,k_i\text{+}2} = \alpha_{v_i,k_i{+}1} + 2\alpha_{v_i,k_i{+}2}\text{.}
    \end{equation*}
    Subsequently, for $j = 1, \ldots, n_i{-}k_i{-}2$, the blow-up $\eta_{v_i,j}$ is centered at $\alpha_{v_i,k_i{+}j}$. Taking their pullbacks successively, we obtain
    \begin{equation*}
        \sum_{j = 1}^{n_i\text{-}k_i\text{-}1} j \alpha_{v_i,k_i{+}j}\text{.}
    \end{equation*}
    The blow-up $\eta_{v_i,n_i\text{-}k_i\text{-}1}$ is centered at $\alpha_{v_i,k_i}$, which is not a component in the canonical divisor. Therefore we only add the exceptional curve $\alpha_{v_i,k_i{-}1}$. For $j = 0, \ldots , k_i-1$, the blow-up $\eta_{v_i,n_{v_i}\text{-}k_i\text{-}1{+}j}$ is centered at $\alpha_{v_i,k_i{-}j}$, and we conclude that the part of $K_X$ supported in $G_{v_i}$ is equal to 
    \begin{equation*}
       \sum_{j=1}^{k_i}j \alpha_{v_i,k_i\text{-}j} + \sum_{j = 1}^{n_i\text{-}k_i\text{-}1} j \alpha_{v_i,k_i{+}j} = \sum_{j=0}^{n_{v_i}\text{-}1}|k_i{-}j|\alpha_{v_i,j} \text{.}
    \end{equation*}
    In order to write $K_X$ in terms of the basis $\cB$, we substitute
    \begin{equation*}
        k_i\alpha_{v_i,0} = k_i(G - \alpha_{v_i,1} - \ldots - \alpha_{v_i,n_{v_i}\text{-}1})
    \end{equation*}
    Thus, we obtain the result.
\end{proof}

\subsection{Mordell--Weil ranks of rational elliptic surfaces via conic bundles}\label{subsection:rank delta}

Let $\pi \colon X \to \PP^1$ be a rational elliptic surface over and $\varphi \colon X \to \PP^1$ a conic bundle over $k$. In this section, we relate the rank $r$ of $\cE(\overline{k}(T))$ to the number $\delta$ of fibers of type $A_n$ of $\varphi$. Firstly, by Corollary \ref{rational_shioda-tate} and Proposition \ref{conic_bundle_configuration}, we know that  both $r$ and $\delta$ are at most 8. The following proposition shows another way in which these numbers are related.

\begin{prop}\label{prop: rank delta 8}
    If $r = 8$, then $\delta = 8$.
\end{prop}

\begin{proof}
    If $r=8$, by Corollary \ref{rational_shioda-tate} $\pi$ has no reducible fibers. Consequently, we can use Proposition \ref{results_RES} to conclude that there are no rational $(-2)$-curves in $X$. By Theorem \ref{CBfibers}, every reducible fiber of $\varphi$ is of type $A_2$, so using Proposition \ref{conic_bundle_configuration} we obtain $\delta = 8$.
\end{proof}

In what follows, we prove that $\delta \geq r$. We start with a useful definition.

\begin{dfn}\label{definition: ell_v}
    Let $F_v$ be a fiber of $\pi \colon X \to \PP^1$. Then, we define
    \begin{equation*}
        \ell_v \colonequals \#\{\Theta \text{ irreducible component of $F_v$}| \, \Theta \text{ is a fiber component of $\varphi$}\}. 
    \end{equation*}
\end{dfn}

Before the main result, we prove the following Lemmata.

\begin{lemma}\label{lemma:shared_components}
    Let $F_v$ be a fiber of $\pi \colon X \to \PP^1$, and let $m_v$ be its number of irreducible components. Then, $\ell_v \leq m_v -1$.
\end{lemma}

\begin{proof}
    By Definition \ref{definition: ell_v}, its clear that $\ell_v \leq m_v$. Suppose that for a fiber $F_v$ we have $\ell_v = m_v$. Then, since every component of $F_v$ is a fiber component of $\varphi$, we have $F_v \cdot G = 0$. This is not possible by Theorem \ref{CBfibers}.
\end{proof}

\begin{lemma}\label{lemma: ell_v}
    We can write the sum of $\ell_v$ for every $v \in \PP^1$ as follows:
    \begin{equation*}
        \sum_{v \in \PP^1} \ell_v = \sum_{i=1}^{\delta} (n_{v_i}-2) + \sum_{i=1}^{\varepsilon}(n_{w_i}-1)   
    \end{equation*}
\end{lemma}

\begin{proof}
    By Proposition \ref{results_RES}, every $(-2)$-curve in $X$ is a fiber component of $\pi$. Therefore, 
    \begin{equation*}
        \sum_{v \in \PP^1} \ell_v = \#\{\Theta \text{ irreducible component of a fiber of $\varphi$}| \, \Theta^2 = -2\}.
    \end{equation*}
    We obtain the result by the classification in Table \ref{CBfibers_table}.
\end{proof}

\begin{lemma}\label{lemma:defect}
    We can write the difference $\delta - r$ as follows:
    \begin{equation*}
        \delta - r = \sum_{v \in \PP^1} (m_v-1 - \ell_v). 
    \end{equation*}
\end{lemma}

\begin{proof}
    By algebraic manipulation we can write
    
    \begin{equation*}
        \sum_{v \in \PP^1} (n_v -1) = \delta + \sum_{i=1}^{\delta}(n_{v_i}-2) + \sum_{i=1}^{\varepsilon} (n_{w_i}-1)\text{.}
    \end{equation*}
    Then, Lemma \ref{lemma: ell_v} yields
    \begin{equation*}
        \sum_{v \in \PP^1} (n_v -1) = \delta + \sum_{v \in \PP^1} \ell_v\text{.} 
    \end{equation*}
    By Proposition \ref{conic_bundle_configuration} and the Shioda--Tate formula (Corollary \ref{rational_shioda-tate}), we have
    \begin{equation*}
        8 = \delta + \sum_{v \in \PP^1} \ell_v = r + \sum_{v \in \PP^1}(m_v-1)\text{.}
    \end{equation*}
    Rearranging, we obtain the result.
\end{proof}

With this, we are ready to prove the result.

\begin{prop}\label{prop:defect is not negative}
    Let $X$ be a rational surface with an elliptic fibration $\pi$ of Mordell--Weil rank $r$ over $\overline{k}$. Let $\varphi \colon X \to \PP^1$  be a conic bundle  with $\delta$ fibers of type $A_n$. Then, $\delta \geq r$.
\end{prop}

\begin{proof}
    Applying Lemma \ref{lemma:shared_components}, we have
    \begin{equation*}
        \sum_{v \in \PP^1} \ell_v \leq \sum_{v \in \PP^1} (m_v-1).
    \end{equation*}
    Thus, Lemma \ref{lemma:defect} yields
    \begin{equation*}
        \delta - r = \sum_{v \in \PP^1} (m_v-1 - \ell_v) \geq 0.
    \end{equation*}
\end{proof}

\begin{remark}
    Notice that Proposition \ref{prop: rank delta 8} is a special case of Proposition \ref{prop:defect is not negative}.
\end{remark}

\section{Conic bundles via Weierstrass Equations}\label{section:CBs_Weierstrass}

In this section, we study the elliptic curves $\cE$ over $k(T)$ defined by Equation \ref{equation1}

\begin{equation}\tag{\ref{equation1}}
    y^2 = a_3(T)x^3 + a_2(T)x^2 + a_1(T)x + a_0(t)\text{,}
\end{equation}

where $\deg a_i(T) \leq 2$, $a_i(T)$ are not all multiple of $(T-c)^2$ for any $c \in \overline{k}$ and $\Deltaell(T)$ not identically 0. Under these assumptions, the Kodaira--Néron model of $\cE$ is a rational elliptic surface $\pi \colon X \to \PP^1$. Indeed, applying the coordinate changes $x \mapsto x/a_3(T)$, $y \mapsto y/a_3(T)$ yields an equation in Weierstrass form, to which we can apply the criterion in \cite[Section 5.13]{schutt-shioda}.

We can rewrite Equation \ref{equation1} to obtain

\begin{equation}\tag{\ref{equation2}}
    y^2 = A(x)T^2 + B(x)T + C(x).
\end{equation}

Let $\Deltaconic(x) \colonequals B(x)^2 - 4A(x)C(x)$. For each $\theta \in \overline{k}$ such that $\Deltaconic(\theta) = 0$, we have a pair of points $P_{\theta}, -P_{\theta} \in \cE(\overline{k}(T))$, where

\begin{equation}\tag{\ref{points from CB}}
    P_{\theta} = \begin{cases}
        \Bigl(\theta, \sqrt{A(\theta)} \bigl(T - \frac{B(\theta)}{2A(\theta)}\bigr)\Bigr) & \text{if $A(\theta) \neq 0$,}\\
        \hfil\Bigl(\theta, \sqrt{C(\theta)}\Bigr) & \text{if $A(\theta) = 0$.}
    \end{cases} 
\end{equation}

Equation \ref{equation2} also determines a linear system of conics in $X$, i.e. a conic bundle $\varphi \colon X \to \PP^1$. Through this section, we follow Notation \ref{notation_CB_fibers} for reducible fibers of conic bundles. We fix the conic bundle $\varphi$ determined by \ref{equation2}, so $\delta$ denotes the number of fibers of $\varphi$ of type $A_n$ and $\varepsilon$ the number of fibers of type $D_n$.

\subsection{The fiber types of a conic bundle on a rational elliptic surface}\label{subsection: determining fiber types}

\sloppy Let ${\varphi \colon X \to \PP^1}$ be the conic bundle as in Equation \ref{equation2}. In what follows, we determine the singular fibers of $\varphi$ and their respective types, as classified in Theorem \ref{CBfibers}.

\begin{prop}\label{determining_conic_fibers}
    \sloppy Let $\Deltaconic(x) = B(x)^2 - 4A(x)C(x)$, and for every $\theta \in \overline{k}$, let ${G_{\theta} \colonequals \varphi^{-1}([\theta{:}1])}$. Then, $G_{\theta}$ is singular if and only if $\Deltaconic(\theta) = 0$. Moreover, assuming $v_{(x-\theta)}(\Deltaconic) = n-1$, the following hold. 
    \begin{itemize}
        \item[1.] If $A(\theta) \neq 0$ or $C(\theta) \neq 0$, then $G_{\theta}$ is of type $A_n$.
        \item[2.] If $A(\theta) = C(\theta) = 0$, then $G_{\theta}$ is of type $D_n$.
    \end{itemize}
\end{prop}

\begin{proof}

Firstly, assume $A(\theta) \neq 0$. Then, applying the change of coordinates $T = T' - \tfrac{B(x)}{2A(x)}$, we obtain the following equation.

\begin{equation*}
    y^2 = A(x)(T')^{2} - \frac{\Deltaconic(x)}{2A(x)}.
\end{equation*}

Testing by partial derivatives, the point $y = T' = 0$ is singular in the special fiber at $(x-\theta)$ if and only if $\Deltaconic(\theta) = 0$.
If $v_{(x-\theta)}(\Deltaconic(x)) = 1$, then $y = T' = z = 0$ is regular in $X$ (see \cite[Corollary 4.2.12]{liu2006}), and $G_{\theta}$ is of type $A_2$, composed of the two lines $y = \sqrt{A(\theta)} T'$ and $y = -\sqrt{A(\theta)} T'$. If $v_{(x-\theta)}(\Deltaconic(x)) = n - 1$ for $n \geq 3$, then the point $y = T' = z = 0$ is a singularity of type $A_{n-2}$ (see \cite[Table 1]{reid_duval}). Then, $G_{\theta}$ is of type $A_n$, with two components coming from the lines $y = \sqrt{A(\theta)} T'$ and $y = -\sqrt{A(\theta)} T'$ and $n-2$ components in the resolution of the singularity at $y = T' = z = 0$.

Assuming $A(\theta) = 0$ and $C(\theta) \neq 0$, we can apply the change of coordinates $T = 1/u$, $y = y'/u$, arriving at the equation

\begin{equation*}
    (y')^2 = C(x)u^2 + B(x)u + A(x). 
\end{equation*}

Thus, the type of $G_{\theta}$ follows by the previous method.

Finally, assume $A(\theta) = C(\theta) = 0$. Then, $\Deltaconic(x) = B(x)^2$. If $B(\theta) \neq 0$, then the special fiber at $(x-\theta)$ is smooth. Otherwise, the special fiber is a non reduced curve given by $y^2 = 0$. By the classification in Theorem \ref{CBfibers}, we know that $G_{\theta}$ is a fiber of type $D_n$ for some $n \geq 3$. By the resolution of the special fiber, $n = v_{(x-\theta)}(\Deltaconic) + 1$.
\end{proof}

Notice that Proposition \ref{determining_conic_fibers} does not determine the type for the fiber at infinity. In order do this, we perform the change of coordinates $x \mapsto 1/s$, $y = y'/s^2$, obtaining the following equation.

\begin{equation}\label{fiber_at_infinity_eq}
    y'^2 = \tilde{A}(s)T^2 + \tilde{B}(s)T + \tilde{C}(s), 
\end{equation}

where $\tilde{A}(s) = s^4A(1/s)$, $\tilde{B}(s) = s^4B(1/s)$, $\tilde{C}(s) = s^4C(1/s)$. Define $\tilde{\Delta}(s) \colonequals s^8\Deltaconic(1/s)$. Since $A$, $B$ and $C$ have degree at most $3$, we know $\tilde{A}(0) = \tilde{B}(0) = \tilde{C}(0) = 0$. An immediate application of Proposition \ref{determining_conic_fibers} to Equation (\ref{fiber_at_infinity_eq}) yields the following.

\begin{prop}\label{fiber_at_infinity}
    Assume $v_s(\tilde{\Delta}) = 8 - \deg(\Deltaconic) = n-1$. Then, the fiber at infinity of $\varphi \colon X \to \PP^1$ is of type $D_n$.
\end{prop}

\begin{remark}
   Notice that we can recover Proposition \ref{conic_bundle_configuration} by counting the number of components of each reducible fiber in Propositions \ref{determining_conic_fibers} and \ref{fiber_at_infinity}.  
\end{remark}

\subsection{The defect of $X$ and the rank of $\cE$ over $\overline{k}$}\label{section:defect_rank}

Let $\cE$ be an elliptic curve over $k(T)$ defined by Equation \ref{equation1}, $\pi \colon X \to \PP^1$ its Kodaira--Néron model and $\varphi \colon X \to \PP^1$ the conic bundle defined by Equation \ref{equation2}. In what follows, we investigate when the rank $r$ of $\cE$ over $\overline{k}(T)$ is equal to the number $\delta$ of fibers of type $A_n$ of $\varphi$. We use the following definition. 

\begin{dfn}\label{definition:defect}
    The \textit{defect} of $\cE$ is defined as the number \begin{equation*}
        \Df(\cE) \colonequals \delta - r.
    \end{equation*}
\end{dfn}

By Proposition \ref{prop:defect is not negative}, $\Df(\cE) \geq 0$ for any $\cE$ defined by Equation \ref{equation1}.  Since the rank $r$ can be determined through a combination of Tate's Algorithm (\cite{tate1975algorithm}) and the Shioda--Tate formula (Corollary \ref{shiodatate}), and $\delta$ can be determined by Proposition \ref{determining_conic_fibers}, we can always calculate the defect of $\cE$.

\begin{example}
    Let $X$ be the rational elliptic surface and $\varphi \colon X \to \PP^1$ the conic bundle given by the following equation.
    \begin{equation*}
        y^2 = (x^2-1)T + x^3 - x + 4.
    \end{equation*}
    By Tate's algorithm, the elliptic fibration has a single reducible fiber of type $I_2^*$, thus by the Shioda--Tate formula, $r = 2$. Applying Proposition \ref{determining_conic_fibers}, $\varphi$ has two fibers of type $A_3$. Then, $\delta = 2$ and $\Df(\cE) = 0$.
\end{example}

In what follows, we apply the theory of Section \ref{section_CB_and_RES} to $\pi \colon X \to \PP^1$ and $\varphi \colon X \to \PP^1$. Let $G_{w_1}, \ldots, G_{w_{\varepsilon}}$ be the fibers of type $D_n$ for some $n \geq 3$. Since by Proposition \ref{fiber_at_infinity} the fiber at infinity is of type $D_n$, we assume without loss of generality that it is $G_{w_1}$. Let $(O)$ be the zero-section of $\pi \colon X \to \PP^1$. By Equation \ref{equation1}, $(O)$ is contained in $G_{w_1}$. Since $(O)^2 = -1$, we conclude that $(O) = \beta_{w_1,n_{w_1}\text{-}1}$.

\begin{prop}\label{conic_fiber_is_trivial}
    Let $G$ be the class of a fiber of $\varphi \colon X \to \PP^1$. Then, $G \in \Triv(\overline{X})$.
\end{prop}

\begin{proof}
    The fiber class $G$ is represented by $G_{w_1}$, and we write
    \begin{equation*}
        G_{w_1} = \beta_{w_1,0} + \beta_{w_1,1} + 2\beta_{w_1,2} + \cdots + 2 \beta_{w_1, n_{\xi}-2} + 2(O). 
    \end{equation*}
    Since $\beta_{w_1,j}$ is a $(-2)$-curve for $j = 0,1, \ldots , n_{w_1}-2$, they are all fiber components of $\pi$. Therefore, $G_{w_1} \in \Triv(\overline{X})$.  
\end{proof}

By Proposition \ref{results_RES}, the $(-1)$-components of singular fibers of $\varphi$ are sections of $\pi$, so they correspond to $\overline{k}(T)$-points of $\cE$. We use the following notation according to the type of fiber of $\varphi$. For $G_{v_i}$ a fiber of type $A_{n_{v_i}}$, we write $\alpha_{v_i,n_{v_i}-1} = (P_i)$ and $\alpha_{v_i,0} = (P_i')$. For $G_{w_i}$ a fiber of type $D_{n_{w_i}}$, we write $\beta_{w_i,n_{w_i}-1} = (Q_i)$.

\begin{corollary}\label{corollary: points are inverse}
    In $\cE(\overline{k}(T))$, $P'_i = - P_i$ for every $i = 1, \ldots, \delta$, and $[2]Q_i = O$ for every $i = 1, \ldots, \varepsilon$.
\end{corollary}

\begin{proof}
    We can write $G_{v_i} \equiv (P_i)+(P'_i) \mod \Triv(\overline{X})$, since $\alpha_{v_i,j}$ is a $(-2)$-curve for $j = 1,\ldots, n_{i}-2$. Therefore, by Proposition \ref{conic_fiber_is_trivial}, $(P_i)+(P'_i) \in \Triv(\overline{X})$.  Under the isomorphism in Theorem \ref{shioda-tate-isom}, we deduce $P_i \oplus P'_i = O$. Doing the same for $G_{w_i}$, we have $2(Q_i) \in \Triv(\overline{X})$, so $[2]Q_i = O$.
\end{proof}

Notice that Corollary \ref{corollary: points are inverse} agrees with the explicit expression for $\overline{k}(T)$-points induced by roots of $\Deltaconic(x)$. Indeed, if $G_{v_i}$ is a fiber of type $A_n$, then by Proposition \ref{determining_conic_fibers}, it is equal to $G_{\theta}$ for some $\theta \in \overline{k}$ such that $A(\theta)$ or $C(\theta)$ are non-zero. By Equation \ref{points from CB}, the points $P_{\theta}$ and $-P_{\theta}$ are on the line $x = \theta$, so they correspond to the $(-1)$-components of $G_{v_i}$. Without loss of generality we can assume $P_i$ corresponds to $P_{\theta}$. Similarly, if $G_{w_i}$ is a fiber of type $D_n$, then it is induced by a root $\theta$ of $\Deltaconic$ such that $A(\theta) = C(\theta) = 0$, so by Equation \ref{points from CB}, $P_{\theta}$ is a point of 2-torsion in $\cE(\overline{k}(T))$ corresponding to $Q_i$. 

\begin{dfn}\label{def_SML}
    Let $S$ be the set $\{P_1, \ldots, P_{\delta}\}$, $M$ the subgroup of $\cE(\overline{k}(T))$ generated by $S$ and $L$ the free $\ZZ$-module generated over $S$.
\end{dfn}
There is a natural surjection $\psi \colon L \to M$ and we have an exact sequence

\begin{equation}\label{exact_sequence}
    0 \to \ker \psi  \to L \xrightarrow{\psi} M \to 0\text{.}
\end{equation}

The elements of $\ker \psi$ are equivalent to linear relations between $P_1, \ldots, P_{\delta}$ in $\cE(\overline{k}(T))$.

\begin{thm}\label{theorem:kernel_equals_defect}
    The rank of $\ker \psi$ as a $\ZZ$-module is equal to $\Df(\cE)$.
\end{thm}

\begin{proof}
    Let $z = \rank(\ker \psi)$. By the exact sequence (\ref{exact_sequence}), we know $z = \delta - \rank M$. Since $M$ is a submodule of $\cE(\overline{k}(T))$, $\rank M \leq r$. Therefore, by Definition \ref{definition:defect},
        \begin{equation}\label{z_geq_Df}
            z \geq \delta - r = \Df(\cE).
        \end{equation}
        
    There are $z$ independent linear relations
    \begin{center}
    $\begin{matrix}
        [a_{1,1}]P_1 & \oplus \cdots \oplus & [a_{1,\delta}]P_{\delta} &= O,\\
        \vdots & & & \\
        [a_{z,1}]P_1 & \oplus \cdots \oplus & [a_{z,\delta}]P_{\delta} &= O.
    \end{matrix}$
    \end{center}

    Let $a_i = \sum_{j=1}^{\delta} a_{i,j}$. By Proposition \ref{silverman_vertical_divisors}, each linear relation corresponds to an independent vertical divisor $a_{i,1}(P_1) + \ldots + a_{i,\delta}(P_{\delta}) - a_i(O)$. We use these divisors to write a set of independent divisors of $\Triv(\overline{X})$:
    \begin{align*}
        &F, (O),\\
        &\alpha_{v_1,1}, \ldots, \alpha_{v_1,n_{v_1}\text{-}2}, \ldots, \alpha_{v_{\delta},1}, \ldots, \alpha_{v_{\delta},n_{v_{\delta}}\text{-}2},\\
        &\beta_{w_1,0}, \ldots, \beta_{w_1,n_{w_1}\text{-}2}, \ldots, \beta_{w_{\varepsilon},0}, \ldots, \beta_{w_{\varepsilon},n_{w_{\varepsilon}}\text{-}2},\\
        &a_{1,1}(P_1) + \cdots + a_{1,\delta}(P_{\delta}) -a_1(O),\\
        &\vdots\\
        &a_{z,1}(P_1) + \cdots + a_{z,\delta}(P_{\delta}) -a_{z}(O).
    \end{align*}
    We can check that the divisors above are linearly independent by writing them in terms of the basis $\cB$ in Proposition \ref{prop_basis_canonical}. By Proposition \ref{results_RES}, $F = -K_X$, and we write $K_X$ in basis $\cB$ in Proposition \ref{prop_basis_canonical}. The components $\alpha_{v_i,j}$ and $\beta_{w_i,j}$ are generators of $\cB$ for $j \geq 1$, and we can write
    \begin{equation*}
        \beta_{w_i, 0} = G - \beta_{w_i,1} - 2\beta_{w_i,2} - \cdots - 2\beta_{w_i,n_{w_i}\text{-}1}.
    \end{equation*}
    Finally, we use the correspondences $(O) = \beta_{w_1,n_{w_1}\text{-}1}$ and $(P_i) = \alpha_{v_i,n_{v_i}\text{-}1}$ for the rest of the divisors. Therefore, we have a total of $2 + \sum_{i=1}^{\delta}(n_{v_i}-2) + \sum_{i=1}^{\varepsilon}(n_{w_i}-1) + z$ independent divisors in $\Triv(\overline{X})$. Thus, we have
    \begin{equation*}
      2 + \sum_{i=1}^{\delta}(n_{v_i}-2) + \sum_{i=1}^{\varepsilon}(n_{w_i}-1) + z \leq \rank(\Triv(\overline{X})) = 2 + \sum_{v \in \PP^1}(m_v -1).
    \end{equation*}
    By Lemma \ref{lemma:defect}, we conclude
    \begin{equation}\label{z_leq_Df}
        z \leq \Df(\cE).
    \end{equation}
    By the inequalities (\ref{z_geq_Df}) and (\ref{z_leq_Df}), we conclude that $z = \Df(\cE)$.
\end{proof}

A direct consequence of the previous theorem is the following.

\begin{corollary}\label{coro:finite index subgroup}
    The points $P_1, \ldots, P_{\delta}$ determined by the conic bundle $\varphi \colon X \to \PP^1$ generate a finite index subgroup of $\cE(\overline{k}(T))$.
\end{corollary}

\subsection{Bounds on the rank of $\cE$ over $k$}\label{section:rank k bounds}

So far, we have worked over the algebraic closure $\overline{k}$ of the field $k$ over which $\pi \colon X \to \PP^1$ and $\varphi \colon X \to \PP^1$ are defined. In Section \ref{section:defect_rank}, we have studied the relation between $r$ and $\delta$. In order to study the rank $r_k$ of $\cE(k(T))$, we define a number $\delta_k$ which will play a similar part.

\begin{dfn}\label{definition: delta_k}
    Let $\cE$ be a curve given by Equation \ref{equation2}. We define $\delta_k$ as
    \begin{equation*}
        \delta_k = \#\Bigl\{[\theta] : \, \Deltaconic(\theta) = 0, \  \begin{array}{l}
            A(\theta)\text{ is a nonzero square in }k(\theta) \text{ or}\\
            A(\theta) = 0 \text{ and } C(\theta) \text{ is a nonzero square in }k(\theta)
        \end{array}\Bigr\},
    \end{equation*}
    where $[\theta]$ denotes the orbit of $\theta$ by the action of $\Galk$.
\end{dfn}

\begin{remark}
    Notice that by Definition \ref{definition: delta_k}, we can rewrite the result of Theorems \ref{rank A=0} and \ref{rank deg(A) = 0} as follows. If $\cE$ is a curve given by an equation of the form \ref{equation_BBD} and $\deg(A) = 0$, then
    
    \begin{equation*}
        r_{\QQ} = \begin{cases}
            \delta_{\QQ} - 1 & \textit{if $A(x) = \mu \in \QQ^2 \setminus \{0\}$,}\\
            \hfil \delta_{\QQ} & \textit{otherwise.}
        \end{cases} 
    \end{equation*}
\end{remark}

We know $\Galk$ acts on $\NS(\overline{X})$ preserving the intersection product. In particular, any automorphism $\sigma \in \Galk$ sends a $(-1)$-component of a fiber of type $A_n$ to another $(-1)$-component of a fiber of type $A_n$. Thus, for any $P_i \in S$, we have $\sigma(P_i) = \pm P_j$. For each $P_i \in S$, let $[P_i]$ be the orbit of $P_i$ by the action of $\Galk$. The point $\sum_{P \in [P_i]}P$ is invariant under $\Galk$, so it is a $k(T)$-point of $\cE$.

\begin{dfn}\label{definition: SML_k}
    Let $[\theta_i]$ be an orbit in the set
    \begin{equation*}
         \#\Bigl\{[\theta] : \, \Deltaconic(\theta) = 0, \  \begin{array}{l}
            A(\theta)\text{ is a nonzero square in }k(\theta) \text{ or}\\
            A(\theta) = 0 \text{ and } C(\theta) \text{ is a nonzero square in }k(\theta)
        \end{array}\Bigr\}.
    \end{equation*}
    Choose an element $\theta_i' \in [\theta_i]$, and let
    \begin{equation*}
        \Sigma_i = \sum_{P \in [P_{\theta_i'}]} P
    \end{equation*}
    We define the set $S_k \colonequals \{ \Sigma_1, \ldots, \Sigma_{\delta_k}\}$. Further, we define $M_k$ as the subgroup of $\cE(k(T))$ generated by $S_k$, and $L_k$ as the free $\ZZ$-module over $S_k$.
\end{dfn}

\begin{remark}\label{remark: choice theta_i}
    On Definition \ref{definition: SML_k}, different choices of $\theta_i' \in [\theta_i]$ may lead to different results for $\Sigma_i$. Specifically, let $\theta_i', \theta_i'' \in [\theta_i]$ and assume that $\sigma(P_{\theta_i'}) = -P_{\theta_i''}$ for some $\sigma \in \Galk$. Then, $\sum_{P \in [P_{\theta_i'}]} P = - \sum_{P \in [P_{\theta_i''}]} P$. Notice that any linear combination $[n_1]\Sigma_1 \oplus \ldots \oplus [n_{\delta_k}]\Sigma_{\delta_k}$ will induce an equivalent linear combination for any other choice of $\theta_i' \in [\theta_i]$, switching the sign of $n_i$ if necessary.
\end{remark}

Before our main Theorem, we prove the following Lemma on the subgroup $M_k$ of $\cE(k(T))$.

\begin{lemma}\label{lemma_Mk}
    Let $M^{\cG}$ be the subgroup of $M$ (see Definition \ref{def_SML}) invariant by $\Galk$. Then, $M^{\cG} = M_k$.
\end{lemma}

\begin{proof}
    By definition, $M_k$ is a subgroup of $M$ invariant by Galois, so $M_k \subset M^{\cG}$. Let $[n_1]P_1 \oplus \ldots \oplus [n_{\delta}]P_{\delta} \in M^{\cG}$. For each $i = 1, \ldots, \delta$, the point $P_i$ is equal to $P_{\theta}$ (see \ref{points from CB}) for some $\theta \in \overline{k}$ such that $\Deltaconic(\theta) = 0$, and one of $A(\theta)$ and $C(\theta)$ is nonzero. Assume $A(\theta) \neq 0$ is not a square in $k(\theta)$. Then,
     \begin{equation*}
        P_i = \bigl(\theta, \sqrt{A(\theta)} \bigl(T - \tfrac{B(\theta)}{2A(\theta)}\bigr)\bigr),
    \end{equation*}
    and the automorphism $\sqrt{A(\theta)} \mapsto -\sqrt{A(\theta)}$ takes $P_i$ to $-P_i$. Similarly, if $A(\theta) = 0$ and $C(\theta)$ is not a square in $k(\theta)$, $\sqrt{C(\theta)} \mapsto - \sqrt{C(\theta)}$ takes $P_i$ to $-P_i$. Since we assume $[n_1]P_1 \oplus \ldots \oplus [n_{\delta}]P_{\delta}$ is invariant under the action of $\Galk$, we conclude $n_i = 0$.

    Now, assume $A(\theta)$ is a nonzero square in $k(\theta)$ or $A(\theta) = 0$ and $C(\theta)$ is a nonzero square in $k(\theta)$. Then, for each $P \in [P_i]$ distinct from $P_i$, we know $P = \pm P_j$ for some $j \neq i$. If $P_j \in [P_i]$, then $n_i = n_j$, and if $- P_j \in [P_i]$, then $n_i = -n_j$. In both cases, either $(\sum_{P \in [P_i]} P) \in S_k$ or $-(\sum_{P \in [P_i]} P) \in S_k$ (see Remark \ref{remark: choice theta_i}). Thus, $[n_1]P_1\oplus \ldots \oplus [n_{\delta}] P_{\delta} \in M_k$, and we conclude $M_k = M^{\cG}$. 
\end{proof}

There is a natural surjection $\psi_k \colon L_k \to M_k$, and we have an exact sequence
    \begin{equation}\label{exact_sequence_k}
        0 \to \ker \psi_k \to L_k \xrightarrow{\psi_k} M_k \to 0\text{.}
    \end{equation}

We use this exact sequence to prove our main result.

\begin{thm}\label{theorem_rank_k}
    Let $r_k$ be the rank of $\cE(k(T))$. Then, $\delta_k \geq r_k \geq \delta_k - \Df(\cE)$.
\end{thm}

\begin{proof}
    By the exact sequence \ref{exact_sequence_k}, we have $\delta_k = \rank M_k + \rank(\ker \psi_k)$. Since by Lemma \ref{lemma_Mk} $M_k$ is a finite index subgroup of $E(k(T))$, we know $\rank(M_k) = r_k$. On the other hand, each linear relation between points of $S_k$ is a linear relation between points of $S$, so by Theorem \ref{theorem:kernel_equals_defect}, we have $\rank(\ker \psi_k) \leq \Df(\cE)$. This proves the result.
\end{proof}

\begin{remark}
    \sloppy Notice Theorem \ref{theorem_rank_k} is a generalization of Theorem \ref{rank deg(A) > 0} to the context of any number field $k$, and allowing $a_3(T) \neq 1$ in Equation \ref{equation1}.
\end{remark}

\section{Computations of the rank}\label{section: computations of the rank}

Let $\cE$ be a curve given by Equation \ref{equation1}, $\pi \colon X \to \PP^1$ its Kodaira--Néron model and $\varphi \colon X \to \PP^1$ the conic bundle induced by Equation \ref{equation2}.
In general, Theorem \ref{theorem_rank_k} shows that calculating $\delta_k$ determines a range of possible values for $r_k$, but not $r_k$ itself. In this section, we explore cases in which we can determine $r_k$ explicitly. Specifically, we recover Theorems \ref{rank A=0} and \ref{rank deg(A) = 0} in the more general context of number fields.

\subsection{Computation of $\Df(\cE)$}\label{section: computation of defect}

Let $G_{\infty}$ be the fiber at infinity of the conic bundle ${\varphi \colon X \to \PP^1}$. By Proposition \ref{fiber_at_infinity}, $G_{\infty}$ is of type $D_n$, for $n \geq 3$. If $n=3$, then there are 2 distinct reducible fibers of $\pi$ with components in common with $G_{\infty}$ (see \cite[Theorem 5.2]{renato}). If $n \geq 4$, only one reducible fiber has a component in common with $G_{\infty}$. In what follows, we prove that the type of $G_{\infty}$ and the Kodaira types of the fibers of $\pi$ with components in common with $G_{\infty}$ are sufficient for determining the defect $\Df(\cE)$.

\begin{thm}\label{theorem: Determine Df}
    Let $\cE$ be a curve given by Equation \ref{equation1}, $\pi \colon X \to \PP^1$ its Kodaira--Néron model and $\varphi \colon X \to \PP^1$ the induced conic bundle. Let $G_{\infty}$ be the fiber at infinity of $\varphi$ of type $D_n$.
    \begin{itemize}
        \item[i)] If $n=3$, then $\Df(\cE)$ is equal to the number of fibers of type $IV$ or $I_m$, $m \geq 3$, which have an irreducible component in common with $G_{\infty}$.
        \item[ii)] If $n \geq 4$, then $\Df(\cE) \leq 1$. Let $F_a$ be the fiber of $\pi$ with components in common with $G_{\infty}$. In particular, $\Df(\cE) = 1$ in three cases:
        \begin{itemize}
            \item[1.] if $G_{\infty}$ is of type $D_4$ and $F_a$ of type $I_m$ for $m \geq 5$;
            \item[2.] if $G_{\infty}$ is of type $D_5$ and $F_a$ of type $I_1^*$;
            \item[3.] if $G_{\infty}$ is of type $D_6$ and $F_a$ of type $IV^*$.
        \end{itemize}
        For all other configurations of $G_{\infty}$ and $F_a$, $\Df(\cE) = 0$.
    \end{itemize}
\end{thm}

\begin{proof}
    For each fiber $F_v$ of $\pi$, recall that we define $\ell_v$ as the number of components of $F_v$ which are also components of a fiber of $\varphi$ (see Definition \ref{definition: ell_v}). Then, by Proposition \ref{lemma:defect} and Definition \ref{definition:defect},
    \begin{equation}\label{equation: ell_v}
        \Df(\cE) = \sum_{v \in \PP^1}(m_v - 1 - \ell_v).
    \end{equation}

    Let $F_v$ be a fiber of $\pi$ which has no components in common with $G_{\infty}$. If $F_v$ is not reducible, then $m_v = 1$ and $\ell_v = 0$. Assume $F_v$ is reducible and let $\Theta_{v,0}$ be its component intersecting the zero-section $(O)$. We can calculate $\Theta_{0,v} \cdot G_{\infty} = 2$, so $\Theta_{v,0}$ is a 2-section of $\varphi$. On the other hand, the remaining $m_v-1$ components of $F_v$ do not intersect $G_{\infty}$, so they are fiber components of $\varphi$. Thus, $\ell_v = m_v - 1$ and $m_v - 1 - \ell_v = 0$. Therefore, in order to determine $\Df(\cE)$, we just need to determine $\ell_v$ for the fibers $F_v$ which have a component in common with $G_{\infty}$.
    
    i) Let $n = 3$ and $F_a$, $F_b$ be the fibers of $\pi$ which have components in common with $G_{\infty}$. Let $\Theta_{a,0}, \Theta_{b,0}$ be the components of $F_a, F_b$ intersecting $(O)$, respectively. We can write $G_{\infty} = \Theta_{a,0} + \Theta_{b,0} + 2(O)$.
    
    If $F_a$ is of type $IV$ or $I_m$, $m \geq 3$, then $\Theta_{a,0}$ intersects $\Theta_{a,1}$ and $\Theta_{a,m_{a{\text{-}1}}}$. Since these components intersect $G_{\infty}$, they are sections of $\varphi$. On the other hand, the remaining $m_a-3$ components $\Theta_{a,2}, \ldots, \Theta_{a,m_a{\text{-}2}}$ do not intersect $G_{\infty}$, so they must be components of a fiber of $\varphi$. Thus, $\ell_a = m_a - 2$. If $F_a$ is of one of the remaining Kodaira types, then $\Theta_{a,0}$ only intersects $\Theta_{a,1}$. The other components $\Theta_{a,2}, \ldots, \Theta_{a,m_a{\text{-}}1}$ do not intersect $G_{\infty}$, so they are fiber components in $\varphi$. Thus, $\ell_a = m_a - 1$. We determine $\ell_b$ by the same arguments. Substituting every $\ell_v$ in Equation \ref{equation: ell_v}, we obtain the result.\\

    ii) Let $n \geq 4$ and $F_a$ be the fiber of $\pi$ which has a component in common with $G_{\infty}$. For each $n$, the Kodaira type of $F_a$ is restricted by the intersection pattern on the $(-2)$-components of $G_{\infty}$. We prove the result through the following steps for every possible combination of types of $G_{\infty}$ and $F_a$.
    \begin{itemize}
        \item[1.] Determine which components of $F_a$ are also components of $G_{\infty}$.
        \item[2.] From the remaining components of $F_a$, we determine which ones intersect $G_{\infty}$. These are section (or multi-sectios) of $\varphi$.
        \item[3.] The remaining components do not intersect $G_{\infty}$. Therefore, they are fiber components of $\varphi$.
    \end{itemize}
    With this, we determine $\ell_a$, so we can calculate $\Df(\cE) = m_a - 1 - \ell_a$. We do not show these steps explicitly, only the final results in Table \ref{Table_configurations}. The first two columns show the types of $G_{\infty}$ and $F_a$. The third column shows the components of $F_a$ which are also fiber components of $\varphi$, and the fourth column shows the components which are sections or multi-sections of $\varphi$, following the notation in \cite[Theorem 5.12]{schutt-shioda}. Finally, the fifth column shows the value of $\Df(\cE)$ for each combination.
    \end{proof}

    \begin{table}[h]
        \centering
        \begin{tabular}{|c|c|c|c|c|}
            \hline
            $G_{\infty}$ & $F_a$ & fiber components of $\varphi$ & (multi)-sections of $\varphi$ & $\Df(\cE)$\\
            \hline
            $D_4$ & $I_4$ & $\Theta_0, \Theta_1, \Theta_3$ & $\Theta_2$ & $0$\\
            \hline
            $D_4$ & $I_{m\geq 5}$ & $\Theta_0, \Theta_1, \Theta_3, \ldots, \Theta_{m{\text{-}}3}, \Theta_{m{\text{-}}1}$ & $\Theta_2, \Theta_{m{\text{-}}2}$ & $1$ \\
            \hline
            $D_5$ & $I_1^*$ & $\Theta_0, \Theta_1, \Theta_4, \Theta_5$ & $\Theta_2, \Theta_3$ & $1$\\
            \hline
            $D_5$ & $I_{m \geq 2}^*$ & $\Theta_0, \ldots, \Theta_5, \Theta_7, \ldots, \Theta_{m+4}$ & $\Theta_6$ & $0$\\
            \hline
            $D_6$ & $IV^*$ & $\Theta_0, \Theta_2, \Theta_3, \Theta_4, \Theta_6$ & $\Theta_1, \Theta_5$ & $1$\\
            \hline
            $D_7$ & $III^*$ & $\Theta_0, \ldots, \Theta_4, \Theta_6,\Theta_7$ & $\Theta_5$ & $0$\\
            \hline
            $D_9$ & $II^*$ & $\Theta_0, \Theta_2, \ldots, \Theta_8$ & $\Theta_1$ & $0$\\
            \hline
            $D_{m+5}$ & $I_{m \geq 0}^*$ & $\Theta_0, \Theta_2, \ldots, \Theta_{m+4}$ & $\Theta_1$ & $0$\\
            \hline
        \end{tabular}
        \caption{$\Df(\cE)$ for each configuration of $G_{\infty}$ and $F_a$}
        \label{Table_configurations}
    \end{table}

In particular, Theorem \ref{theorem: Determine Df} shows that $\Df(\cE) \leq 2$ for every curve $\cE$ given by Equation \ref{equation1}.

\subsection{Families of curves with $\Df(\cE) = 0$}\label{section: Df = 0}

If $\cE$ is a curve given by Equation \ref{equation1}, then by Theorem \ref{theorem_rank_k}, we know $\delta_k \geq r_k \geq \delta_k - \Df(\cE)$. In general, this is not enough to determine $r_k$ explicitly. The obvious exceptions are the examples in which $\Df(\cE) = 0$. In this section, we use Theorem \ref{theorem: Determine Df} to find the families of curves for which every member has $\Df(\cE) = 0$, and thus $r_k = \delta_k$.

We start by looking at the family of curves given by Equation \ref{equation1} in which $a_3(T)$ is non-constant.

\begin{thm}\label{theorem: generic is 2I2}
    Let $\cE$ be a curve given by Equation \ref{equation1}, and $\gamma(T) \colonequals \Deltaell(T)/a_3(T)^2$. Assume that $\deg(a_3) \geq 1$, $\deg(\gamma) = 8$ and that the resultant  $\Res(a_3,\gamma)$ is nonzero. Then, $r_k = \delta_k$.
\end{thm}

\begin{proof}
    Firstly, we put Equation \ref{equation1} in Weierstrass form by applying the coordinate changes $x \mapsto x/a_3(T)$ and $y \mapsto y/a_3(T)$. After clearing the denominator, we have
    \begin{equation*}
        y^2 = x^3 + a_2(T)x^2 + a_1(T)a_3(T)x + a_0(T)a_3(T)^2.
    \end{equation*}
    The discriminant of $\cE$ is given by
    \begin{align*}
        \Deltaell(T) &= -27 a_0^2a_3^4 + 18a_0a_1a_2a_3^3 + a_1^2a_2^2a_3^2 - 4a_0a_2^3a_3^2 - 4 a_1^3a_3^3\\
        &= a_3^2(-27a_0^2a_3^2 + 18 a_0a_1a_2a_3 + a_1^2a_2^2 - 4a_0a_2^3 - 4a_1^3a_3)\\
        &= a_3^2(T) \gamma(T).
    \end{align*}
    Let $G_{\infty}$ be the fiber at infinity of the conic bundle $\varphi \colon X \to \PP^1$. Recall that by Proposition \ref{fiber_at_infinity}, $G_{\infty}$ is of type $D_n$ where $n = 9 - \deg(\Deltaconic)$.

    Assume $a_3(T) = p(T-q)$, where $p \in k^{\times}$ and $q \in k$. Then, writing $\cE$ in the form of Equation \ref{equation2}, we have $\deg(A) \leq 2$, $\deg(B) = 3$. Thus, $\deg(\Deltaconic) = 6$ and $G_{\infty}$ is of type $D_3$. There are two distinct fibers of $\pi \colon X \to \PP^1$ which have components in common with $G_{\infty}$, namely, the fiber $F_q$ at $T = q$ and the fiber at infinity $F_{\infty}$. We can use Tate's Algorithm to determine the fiber types of $F_q, F_{\infty}$. Since $\Res(a_3,\gamma) = 0$, $q$ is not a root of $\gamma(T)$, so $F_q$ is of type $I_2$. Similarly, since $\deg(\gamma) = 8$, $F_{\infty}$ is of type $I_2$. By Theorem \ref{theorem: Determine Df}, $\Df(\cE) = 0$.

    Assume $a_3(T) = p(T-q_1)(T-q_2)$, where $p \in k^{\times}$ and $q_i \in \overline{k}$. Calculating $\Deltaconic$, we obtain the lead coefficient in $x_6$ equals $p^2(q_1-q_2)^2$. If $q_1 = q_2$, then $\deg(\Deltaconic) = 6$ and $G_{\infty}$ is of type $D_3$. The fibers $F_{q_1}$ and $F_{q_2}$ have components in common with $G_{\infty}$, and by Tate's Algorithm both are of type $I_2$. If $q_1 = q_2$, then $\deg(\Deltaconic) \leq 5$ and $G_{\infty}$ is of type $D_n$ for some $n \geq 4$. By Tate's Algorithm, the fiber $F_{q_1}$ is of type $I_4$. By Table \ref{Table_configurations}, $G_{\infty}$ is of type $D_4$. In both cases, by Theorem \ref{theorem: Determine Df}, $\Df(\cE) = 0$, so by Theorem \ref{theorem_rank_k} $r_k = \delta_k$.
\end{proof}

Notice that the conditions imposed in Theorem \ref{theorem: generic is 2I2} on the coefficients of $a_i(T)$ exclude only a Zariski closed set. In this sense, this Theorem implies that in the family of curves given by Equation \ref{equation1}, a general member $\cE$ has $r_k = \delta_k$.

\begin{example}\label{example1}
    Let $\cE$ be given by
    \begin{align}\label{equation: example}
        y^2 &= Tx^3 + (T^2+aT+b)x^2 + (cT^2 + dT + e)x + (fT^2+gT+h)\\
        \nonumber &= (x^2+c x +f)T^2 + (x^3 + ax^2 + dx + g)T + (bx^2 + ex + h),
    \end{align}
    for $a,b,c,d,e,f,g,h \in k$. The condition $\Res(a_3,\gamma) \neq 0$ is equivalent to ${\gamma(0) = b^2(e^2 - 4bh) \neq 0}$. The condition $\deg(\gamma) = 8$ is equivalent to the coefficient of $\gamma$ in $T^8$ being nonzero, that is, $c^2-4f \neq 0$. Under these assumptions, $r_k = \delta_k$.
\end{example}

In what follows, we apply Theorem \ref{theorem: Determine Df} to recover previous results stated in Section \ref{section:survey}.

\begin{prop}\label{recover_ALRM}
    Let $\cE$ be a curve given by the following equation, for $A,B,C,D,a,b,c \in k$
    \begin{align}\label{alrm_equation_2}
        \cE \colon  \ y^2 &= x^3T^2 + 2g(x)T -h(x), \text{ where}\\
        \nonumber g(x) &= x^3 + ax^2 + bx + c, \, c \neq 0;\\
        \nonumber h(x) &= (A-1)x^3 + Bx^2 + Cx + D.
    \end{align}
    Assume $\Deltaconic(x)$ has 6 distinct nonzero roots which are perfect squares over $k$. Then, $r_k = 6$.
\end{prop}

\begin{proof}
    Firstly, notice that writing $\cE$ in the form of Equation \ref{equation2}, we have $A(x) = x^3$. Since by assumption every root of $\Deltaconic$ is a perfect square, we have that $\delta_k = 6$.
    By Proposition \ref{fiber_at_infinity}, the fiber $G_{\infty}$ has type $D_3$, and by Tate's Algorithm, the Kodaira--Néron model $\pi \colon X \to \PP^1$ has two reducible fibers of type $I_2$. Therefore, Theorem \ref{theorem: Determine Df} tells us that $\Df(\cE) = 0$. Applying Theorem \ref{theorem_rank_k}, we obtain the result.
\end{proof}

This result recovers the calculation of the rank in Theorem \ref{ALRM}. We turn next to curves in which $a_3(T)$ is constant in Equation \ref{equation1} and $A(x) = 0$ in Equation \ref{equation2}.

\begin{prop}\label{prop: A = 0}
    Let $\cE$ be a curve given by
    \begin{equation}\label{equation: A = 0}
        \cE \colon \ y^2 = B(x) T + C(x),
    \end{equation}
    where $\deg(B) \leq 2$ and $\deg(C) = 3$. Then, $r_k = \delta_k$.
\end{prop}

\begin{proof}
    Let $G_{\infty}$ be the fiber of $\varphi$ at infinity. Since $\Deltaconic(x) = B(x)^2$, by Proposition \ref{fiber_at_infinity}, $G_{\infty}$ is of type $D_n$ with $n = 9-2\deg(B)$.
    
    Assume $\deg(B) = 0$ or $\deg(B) = 1$. Then $G_{\infty}$ is of type $D_9$ or $D_7$, respectively. By Theorem \ref{theorem: Determine Df}, $\Df(\cE) = 0$ irrespective of the type of fiber of $\pi$ which has components in common with $G_{\infty}$ (see Table \ref{Table_configurations}).

    Assume $\deg(B) = 2$. Then, $G_{\infty}$ is of type $D_5$. By Theorem \ref{theorem: Determine Df}, $\Df(\cE) = 1$ if and only if $G_{\infty}$ has components in common with a fiber of $\pi$ of type $I_1^*$. Writing  $\cE$ in the form of Equation \ref{equation1}, we have that $a_3(T)$ is a nonzero constant, and $\deg(a_i) \leq 1$ for $i = 0,1,2$. Then, $\deg(\Deltaell) \leq 3$. By Tate's Algorithm, we deduce $\pi$ has a fiber of additive reduction with at least 7 components. Thus, $G_{\infty}$ has components in common with a fiber of type $I_m^*$ for $m \geq 2$, so $\Df(\cE) = 0$.

    Applying Theorem \ref{theorem_rank_k}, we deduce that $r_k = \delta_k$.
\end{proof}

This result generalizes Theorem \ref{rank A=0} to a general number field $k$.

\subsection{Families of curves with $\Df(\cE) > 0$} \label{section: Df > 0}

In this section we study families of curves $\cE$ given by Equation \ref{equation1} for which $\Df(\cE) > 0$. Then, Theorem \ref{theorem_rank_k} is not enough to determine the rank $r_k$. We explore cases in which we can use additional information to determine $r_k$.

\begin{prop}\label{prop: Df = 1}
    Let $\cE$ be a curve given by
    \begin{equation}\label{equation Df = 1}
        y^2 = \mu T^2 + B(x)T + C(x),
    \end{equation}
    where $\mu \in k^{\times}$, $\deg(B) \leq 2$, $\deg(C) = 3$. Then, $\Df(\cE) = 1$.
\end{prop}

\begin{proof}
     By Proposition \ref{fiber_at_infinity}, the conic fiber $G_{\infty}$ at infinity is of type $D_n$ where ${n = 9 - \deg(\Deltaconic)}$. Let $F_a$ be the fiber of $\pi$ with components in common with $G_{\infty}$.
    
    Assume $\deg(B) \leq 1$. Then $G_{\infty}$ is of type $D_6$. By Table \ref{Table_configurations}, $F_a$ is of type $IV^*$ or $I_1^*$. Writing $\cE$ in the form of Equation \ref{equation1}, we have $\deg(a_0) = 2$, $\deg(a_1) \leq 1$ and $\deg(a_2) = \deg(a_3) = 0$. Therefore, $\deg(\Deltaell) = 4$, and by Tate's Algorithm, $\pi$ has fiber of additive reduction at infinity with $7$ components. Thus, $F_a$ is of type $IV^*$, and by Theorem \ref{theorem: Determine Df}, $\Df(\cE) = 1$.

    Now, assume $\deg(B) = 2$. Then $G_{\infty}$ is of type $D_5$, and $F_a$ is of type $I_m^*$ for some $m \geq 0$. Similarly, writing $\cE$ in the form of Equation \ref{equation1}, we have $\deg(a_0) = 2$, $\deg(a_1) \leq 1$, $\deg(a_2) = 1$ and $\deg(a_3) = 0$. Therefore $\deg(\Deltaell) = 5$ and by Tate's Algorithm there is a fiber of additive reduction at infinity with $6$ components. Thus $F_a$ is of type $I_1^*$ and by Theorem \ref{theorem: Determine Df}, $\Df(\cE) = 1$.
\end{proof}

In this situation, Theorem \ref{theorem_rank_k} is not enough to determine the rank $r_k$ of $\cE$.

\begin{prop}\label{prop: A = mu}
    Let $\cE$ be a curve given by Equation \ref{equation Df = 1}. Then,
    \begin{equation*}
        r_k = \begin{cases}
            \delta_k - 1 & \text{if $\mu \in k^2$,}\\
            \hfil \delta_k & \textit{if $\mu \in k \setminus k^2$.}
        \end{cases}
    \end{equation*}
\end{prop}

\begin{proof}
    For every $\theta \in \overline{k}$ such that $\Deltaconic(\theta) = 0$, let $P_{\theta}$ be the induced point in $\cE(\overline{k}(T))$ (see \ref{points from CB}). Recall from Definition \ref{def_SML} that $S$ is the set of points $P_{\theta}$. By \cite[Proposition 12]{battistoni2021ranks}, we know that
    \begin{equation}\label{equation: linear relation}
        \sum_{\theta: \Delta(\theta) = 0} [n_{\theta}]P_{\theta} = O \in \cE(\overline{k}(T)), 
    \end{equation}
    where $n_i = v_{(x-\theta)}(\Deltaconic)$. Notice that since $\Deltaconic(x) \in k[x]$, we have that $n_{\theta'} = n_{\theta}$ if $\theta' \in [\theta]$.
    By Proposition \ref{prop: Df = 1}, $\Df(\cE) = 1$. Therefore, Theorem \ref{theorem:kernel_equals_defect} implies that Equation \ref{equation: linear relation} is the only linear relation between points of $S$, up to scalar multiplication.

    Assume $\mu$ is a square in $k(\theta)$. Then, $\sqrt{\mu} \in k(\theta)$. For any $\sigma \in \Galk$, we have
    \begin{equation*}
        \sigma(P_{\theta}) = \begin{cases}
            \hfil\bigl(\sigma(\theta), \sqrt{\mu}(T+\tfrac{B(\sigma(\theta))}{2\mu})\bigr) = P_{\theta} & \text{if $\sigma \in \Gal(\overline{k}/k(\sqrt{\mu}))$,}\\
            \bigl(\sigma(\theta), - \sqrt{\mu}(T+\tfrac{B(\sigma(\theta))}{2\mu})\bigr) = - P_{\theta} & \text{if $\sigma \not\in \Gal(\overline{k}/k(\sqrt{\mu}))$.}
        \end{cases}
    \end{equation*}
    
    If $\mu \in k^2$, then $k(\sqrt{\mu}) = k$. Thus, for each $\Sigma_i \in S_k$ we can write $\Sigma_i = \sum_{\theta' \in [\theta]}P_{\theta'}$ for some $\theta \in \overline{k}$ such that $\Deltaconic(\theta) = 0$. Then, Equation \ref{equation: linear relation} determines a linear relation between points of $S_k$. By Theorem $\ref{theorem_rank_k}$, $r_k = \delta_k -1$.

    If $\mu \in k \setminus k^2$, then for each $\Sigma_i \in S_k$, $\Sigma_i$ is the sum of $P_{\theta'}$ for half of $\theta' \in [\theta]$, and $-P_{\theta}$ for the other half. Then, any linear relation between point of $S_k$ induces a linear relation between points of $S$ strictly different from Equation \ref{equation: linear relation}. Since $\Df(\cE) = 1$, this is not possible, so $r_k = \delta_k$. 
\end{proof}

Notice that this result generalizes Theorem \ref{rank deg(A) = 0} to any number field.

In what follows, we go back to curves given by Equation \ref{equation: A = 0}. We see that if we allow $B(x)$ to be a polynomial of degree 3, then the result of Proposition \ref{prop: A = 0} no longer holds in general.

\begin{prop}\label{prop: A=0, deg(B)=3}
    Let $\cE$ be a curve given by
     \begin{equation}\tag{\ref{equation: A = 0}}
        \cE \colon \ y^2 = B(x) T + C(x),
    \end{equation}
    with $B(x) = x^3 + ax^2 + bx + c$ is a separable polynomial of degree 3 and $C(x) =  \lambda B(x) + \mu$, for $\lambda \in k$, $\mu \in k^{\times}$. Then, $\Df(\cE) = 1$, and
    \begin{equation*}
        r_k = \begin{cases}
            \delta_k -1 & \textit{if $\mu \in k^2$,}\\
            \hfil \delta_k & \textit{if $\mu \in k \setminus k^2$.}
        \end{cases}
    \end{equation*}
\end{prop}

\begin{proof}
    Since $\deg(\Deltaconic) = 6$, the fiber $G_{\infty}$ of $\varphi$ at infinity is of type $D_3$. Writing $\cE$ in the form of Equation \ref{equation1}, we have
    \begin{equation*}
        y^2 = (T + \lambda)x^3 + a(T+\lambda)x^2 + b(T+\lambda)x + c(T+\lambda) + \mu.
    \end{equation*}
    The fibers of $\pi$ which have components in common with $G_{\infty}$ are the fiber $F_{-\lambda}$ at $T = - \lambda$ and $F_{\infty}$ at infinity. We can calculate that $v_{(T-\lambda)}(\Deltaell) = 4$, and since  $(T-\lambda)$ divides $a_2$, by Tate's Algorithm, $F_{-\lambda}$ is of type $IV$.
    
    On the other hand, by calculating the lead coefficient of $\Deltaell(T)$, we have that $\deg(\Deltaell) = 6$ if and only if $B(x)$ is separable. Then, by Tate's Algorithm, $F_{\infty}$ is of type $I_0^*$. Thus, by Theorem \ref{theorem: Determine Df}, $\Df(\cE) = 1$.

    Let $\theta_1, \theta_2, \theta_3$ denote the roots of $B(\theta)$. Then, the point corresponding to $\theta_i$ in $\cE(\overline{k}(T))$ is $P_{i} = (\theta_i, \sqrt{\mu})$. By this equation, we have $P_1 \oplus P_2 \oplus P_3 = O$. We prove the formula for $r_k$ by using the same arguments as in Proposition \ref{prop: A = mu}.
\end{proof}

Finally, we return to Example \ref{example1}. Modifying the equation, we provide an example of a family of curves with $\Df(\cE) = 2$.

\begin{example}\label{example2}
    Let $\cE$ be given by
    \begin{equation*}
        y^2 = Tx^3 + (T^2+aT+1)x^2 + (2bT^2 + cT + 2d)x + (b^2T^2+eT+d^2),
    \end{equation*}
    for $a,b,c,d,e,f \in k$, and let $\gamma(T) = \Deltaell(T)/T^2$. By our choice of coefficients, $\deg(\gamma(T)) \leq 7$, and $\gamma(0) = 0$.
    The fibers $F_0$ and $F_{\infty}$ are the fibers of $\pi$ which have components in common with $G_{\infty}$, and by Tate's algorithm, both are of type $I_m$ for some $m \geq 3$. By Theorem \ref{theorem: Determine Df}, $\Df(\cE) = 2$.
\end{example}

\bibliographystyle{amsalpha}
\bibliography{bibliography}

\end{document}

%% file: sample.tikzstyles
\tikzstyle{whitedot}=[fill=white, draw=black, shape=circle, inner sep=2pt]
\tikzstyle{blackdot}=[fill=black, draw=black, shape=circle, inner sep=2pt]
\tikzstyle{smallblackdot}=[fill=black, draw=black, shape=circle, inner sep=0pt]

\tikzstyle{red}=[-, draw=red]
\tikzstyle{green}=[-, draw={rgb,255: red,0; green,217; blue,0}]
\tikzstyle{blue}=[-, draw=blue]
\tikzstyle{thick}=[-, line width=0.6mm]
\tikzstyle{arrow}=[->]

%% file: conicfiberA2.tikz
\begin{tikzpicture}
	\begin{pgfonlayer}{nodelayer}
		\node [style=whitedot] (0) at (-1, 0) {};
		\node [style=whitedot] (1) at (1, 0) {};
		\node [style=none] (2) at (-1, -0.75) {\tiny{1}};
		\node [style=none] (3) at (1, -0.75) {\tiny{1}};
	\end{pgfonlayer}
	\begin{pgfonlayer}{edgelayer}
		\draw (1) to (0);
	\end{pgfonlayer}
\end{tikzpicture}

%% file: conicfiberAn.tikz
\begin{tikzpicture}
	\begin{pgfonlayer}{nodelayer}
		\node [style=whitedot] (0) at (-2, 0) {};
		\node [style=whitedot] (1) at (2, 0) {};
		\node [style=blackdot] (2) at (-1, 0) {};
		\node [style=blackdot] (3) at (1, 0) {};
		\node [style=none] (4) at (-0.5, 0) {};
		\node [style=none] (5) at (0.5, 0) {};
		\node [style=smallblackdot] (6) at (-0.25, 0) {};
		\node [style=smallblackdot] (7) at (0, 0) {};
		\node [style=smallblackdot] (8) at (0.25, 0) {};
		\node [style=none] (9) at (-2, -0.75) {\tiny{1}};
		\node [style=none] (10) at (-1, -0.75) {\tiny{1}};
		\node [style=none] (11) at (1, -0.75) {\tiny{1}};
		\node [style=none] (12) at (2, -0.75) {\tiny{1}};
	\end{pgfonlayer}
	\begin{pgfonlayer}{edgelayer}
		\draw (0) to (2);
		\draw (2) to (4.center);
		\draw (5.center) to (3);
		\draw (3) to (1);
	\end{pgfonlayer}
\end{tikzpicture}

%% file: conicfiberD3.tikz
\begin{tikzpicture}
	\begin{pgfonlayer}{nodelayer}
		\node [style=blackdot] (0) at (-1, 0) {};
		\node [style=blackdot] (1) at (1, 0) {};
		\node [style=whitedot] (2) at (0, 0) {};
		\node [style=none] (3) at (-1, -0.75) {\tiny{1}};
		\node [style=none] (4) at (0, -0.75) {\tiny{2}};
		\node [style=none] (5) at (1, -0.75) {\tiny{1}};
	\end{pgfonlayer}
	\begin{pgfonlayer}{edgelayer}
		\draw (0) to (2);
		\draw (2) to (1);
	\end{pgfonlayer}
\end{tikzpicture}

%% file: conicfiberDn.tikz
\begin{tikzpicture}
	\begin{pgfonlayer}{nodelayer}
		\node [style=blackdot] (0) at (-1.75, 0.75) {};
		\node [style=blackdot] (1) at (-1, 0) {};
		\node [style=blackdot] (2) at (-1.75, -0.75) {};
		\node [style=none] (3) at (-0.5, 0) {};
		\node [style=smallblackdot] (4) at (-0.25, 0) {};
		\node [style=smallblackdot] (5) at (0, 0) {};
		\node [style=smallblackdot] (6) at (0.25, 0) {};
		\node [style=none] (7) at (0.5, 0) {};
		\node [style=blackdot] (8) at (1, 0) {};
		\node [style=whitedot] (9) at (2, 0) {};
		\node [style=none] (10) at (-2.5, 0.75) {\tiny{1}};
		\node [style=none] (11) at (-2.5, -0.75) {\tiny{1}};
		\node [style=none] (12) at (-1, -0.75) {\tiny{2}};
		\node [style=none] (13) at (1, -0.75) {\tiny{2}};
		\node [style=none] (14) at (2, -0.75) {\tiny{2}};
	\end{pgfonlayer}
	\begin{pgfonlayer}{edgelayer}
		\draw (2) to (1);
		\draw (0) to (1);
		\draw (1) to (3.center);
		\draw (8) to (9);
		\draw (8) to (7.center);
	\end{pgfonlayer}
\end{tikzpicture}

%% file: whitedot.tikz
\begin{tikzpicture}
	\begin{pgfonlayer}{nodelayer}
		\node [style=whitedot] (0) at (0, 0) {};
	\end{pgfonlayer}
\end{tikzpicture}

%% file: blackdot.tikz
\begin{tikzpicture}
	\begin{pgfonlayer}{nodelayer}
		\node [style=blackdot] (0) at (0, 0) {};
	\end{pgfonlayer}
\end{tikzpicture}

%% file: blowdown_to_standard_An.tikz
\begin{tikzpicture}
	\begin{pgfonlayer}{nodelayer}
		\node [style=whitedot] (0) at (-7, 0) {};
		\node [style=whitedot] (1) at (7, 0) {};
		\node [style=blackdot] (2) at (-5, 0) {};
		\node [style=blackdot] (3) at (5, 0) {};
		\node [style=none] (4) at (-4.5, 0) {};
		\node [style=none] (5) at (4.5, 0) {};
		\node [style=smallblackdot] (6) at (-4.25, 0) {};
		\node [style=smallblackdot] (7) at (-4, 0) {};
		\node [style=smallblackdot] (8) at (-3.75, 0) {};
		\node [style=none] (9) at (-7, -1.25) {$\alpha_{v_i,0}$};
		\node [style=blackdot] (13) at (-1, 0) {};
		\node [style=blackdot] (15) at (1, 0) {};
		\node [style=blackdot] (16) at (3, 0) {};
		\node [style=none] (17) at (3.5, 0) {};
		\node [style=smallblackdot] (18) at (3.75, 0) {};
		\node [style=smallblackdot] (19) at (4, 0) {};
		\node [style=smallblackdot] (20) at (4.25, 0) {};
		\node [style=blackdot] (24) at (-3, 0) {};
		\node [style=none] (26) at (-3.5, 0) {};
		\node [style=none] (28) at (-1, -1.25) {$\alpha_{v_i,k_i}$};
		\node [style=none] (31) at (7.75, -1.25) {$\alpha_{v_i,n_{v_i}{-}1}$};
		\node [style=none] (32) at (-7, -0.5) {};
		\node [style=none] (33) at (-7, 0.5) {};
		\node [style=none] (34) at (-3, 0.5) {};
		\node [style=none] (35) at (-3, -0.5) {};
		\node [style=none] (36) at (3, -0.5) {};
		\node [style=none] (37) at (3, 0.5) {};
		\node [style=none] (38) at (7, 0.5) {};
		\node [style=none] (39) at (7, -0.5) {};
		\node [style=none] (40) at (-5, 1.25) {\tiny $k_i$ components};
		\node [style=none] (41) at (5, 1.25) {\tiny $n_{v_i}{-}k_i{-}2$ components};
		\node [style=none] (42) at (0, -2) {};
		\node [style=none] (43) at (0, -3) {};
		\node [style=whitedot] (44) at (-1, -4) {};
		\node [style=whitedot] (45) at (1, -4) {};
		\node [style=none] (46) at (2, -2.5) {\tiny contraction};
	\end{pgfonlayer}
	\begin{pgfonlayer}{edgelayer}
		\draw (0) to (2);
		\draw (2) to (4.center);
		\draw (5.center) to (3);
		\draw (3) to (1);
		\draw (13) to (15);
		\draw (15) to (16);
		\draw (16) to (17.center);
		\draw (24) to (13);
		\draw (26.center) to (24);
		\draw [style=blue, bend right=90, looseness=2.00] (33.center) to (32.center);
		\draw [style=blue] (32.center) to (35.center);
		\draw [style=blue] (33.center) to (34.center);
		\draw [style=blue, bend left=90, looseness=2.00] (34.center) to (35.center);
		\draw [style=blue, bend right=90, looseness=2.00] (37.center) to (36.center);
		\draw [style=blue] (36.center) to (39.center);
		\draw [style=blue] (37.center) to (38.center);
		\draw [style=blue, bend left=90, looseness=2.00] (38.center) to (39.center);
		\draw [style=arrow] (42.center) to (43.center);
		\draw (44) to (45);
	\end{pgfonlayer}
\end{tikzpicture}

%% file: blowdown_to_standard_Dn.tikz
\begin{tikzpicture}
	\begin{pgfonlayer}{nodelayer}
		\node [style=blackdot] (0) at (-2.75, 1) {};
		\node [style=blackdot] (1) at (-1, 0) {};
		\node [style=blackdot] (2) at (-2.75, -1) {};
		\node [style=none] (3) at (-0.5, 0) {};
		\node [style=smallblackdot] (4) at (-0.25, 0) {};
		\node [style=smallblackdot] (5) at (0, 0) {};
		\node [style=smallblackdot] (6) at (0.25, 0) {};
		\node [style=none] (7) at (0.5, 0) {};
		\node [style=blackdot] (8) at (1, 0) {};
		\node [style=whitedot] (9) at (3, 0) {};
		\node [style=none] (10) at (5.5, 0) {$\beta_{w_i,n_{w_i}{-}1}$};
		\node [style=none] (11) at (-1, 0.5) {};
		\node [style=none] (12) at (3, 0.5) {};
		\node [style=none] (13) at (3, -0.5) {};
		\node [style=none] (14) at (-1, -0.5) {};
		\node [style=none] (15) at (1, 1.25) {\tiny $n_{w_i}{-}2$ components};
		\node [style=none] (16) at (0, -1.25) {};
		\node [style=none] (17) at (0, -2.25) {};
		\node [style=whitedot] (18) at (-1, -3) {};
		\node [style=whitedot] (19) at (1, -3) {};
		\node [style=none] (20) at (1.75, -1.75) {\tiny contraction};
		\node [style=none] (21) at (-4, 1) {$\beta_{w_i,0}$};
		\node [style=none] (22) at (-4, -1) {$\beta_{w_i,1}$};
	\end{pgfonlayer}
	\begin{pgfonlayer}{edgelayer}
		\draw (2) to (1);
		\draw (0) to (1);
		\draw (1) to (3.center);
		\draw (8) to (9);
		\draw (8) to (7.center);
		\draw [style=blue] (14.center) to (13.center);
		\draw [style=blue, bend right=90, looseness=1.75] (13.center) to (12.center);
		\draw [style=blue] (12.center) to (11.center);
		\draw [style=blue, bend right=90, looseness=1.75] (11.center) to (14.center);
		\draw [style=arrow] (16.center) to (17.center);
		\draw (18) to (19);
	\end{pgfonlayer}
\end{tikzpicture}